\documentclass[sort&compress]{elsarticle}
\usepackage{amsmath,amssymb,amsthm,multirow,nicematrix,tikz,rotating,enumitem,cases,cleveref,varwidth}
\usepackage{algpseudocodex}
\allowdisplaybreaks[4]

\newtheorem{definition}{Definition}
\newtheorem{lemma}{Lemma}
\newtheorem{problem}{Problem}
\newtheorem{remark}{Remark}
\newtheorem{theorem}{Theorem}
\newtheorem{prop}{Proposition}

\newcommand\adots{\mathinner{\mkern2mu%
	\raisebox{0.1em}{.}\mkern2mu\raisebox{0.4em}{.}%
	\mkern2mu\raisebox{0.7em}{.}\mkern1mu}}
\newcommand\Ao{A_\mathrm o}

\newcommand\faii{\varphi_\mathrm i}

\newcommand\GAmma{\mathit\Gamma}
\newcommand\nsharp{n_\mathrm i^*}
\newcommand\nocstar{n_\mathrm{oc}^*}
\newcommand\Pio{\mathit\Omega_\mathrm o}
\newcommand\Qi{Q_\mathrm i}
\newcommand\Qoc{Q_\mathrm{oc}}
\newcommand\rank{\mathrm{rank}}
\newcommand\rref{\mathrm{rref}}
\newcommand\im{\mathrm{im}}
\newcommand\imr[1]{\mathrm{im}(#1^\mathrm T)}
\newcommand\R{\mathbf R}
\newcommand\T{\mathrm T}
\newcommand\PHi{\mathit\Phi}
\newcommand\PSi{\mathit\Psi}

\journal{XXX}

\begin{document}

\begin{frontmatter}

\title{Determination of the output injection of
  invertible linear systems
  for guaranteeing a regular relative degree}

\tnotetext[ack]{ }

\author[hit]{Shida~Cao}
\ead{shida\_cao@163.com}

\author[hit,keylab]{Bin~Zhou\corref{cor}}
\ead{binzhou@hit.edu.cn}
\cortext[cor]{Corresponding author.}

\author[hit,sust]{Guangren~Duan}
\ead{g.r.duan@hit.edu.cn}

\affiliation[hit]{
	organization = {Center for Control Theory and Guidance Technology},
	addressline  = {Harbin Institute of Technology},
	city         = {Harbin},
	postcode     = {150001},
	state        = {Heilongjiang},
	country      = {People's Republic of China}}

\affiliation[keylab]{
	organization = {National Key Laboratory of Complex System Control
	                and Intelligent Agent Cooperation},
	city         = {Harbin},
	postcode     = {150001},
	state        = {Heilongjiang},
	country      = {People's Republic of China}}

\affiliation[sust]{
	organization = {Shenzhen Key Laboratory of Control Theory and Intelligent Systems},
	addressline  = {Southern University of Science and Technology},
	city         = {Shenzhen},
	postcode     = {518055},
	state        = {Guangdong},
	country      = {People's Republic of China}}

\begin{abstract}
For an invertible linear time-invariant system, the Morse normal form reveals that
there exist an output injection and an output transformation such that the resulting
system possesses a regular relative degree.
Based on this fact, this paper develops a method to determine all admissible output
injections and output transformations that achieve this property.
In contrast to the recursive and implicit computation methods in the existing
literature, the proposed determination method for the output injections is explicit
and easy to use, since the solutions are obtained by solving linear matrix equations.
A rigorous proof of the proposed method is provided, and an illustrative example is
presented to demonstrate the applicability of the approach.
\end{abstract}


\end{frontmatter}

\section{Introduction}

Morse normal form (\cite{Morse1973SIAM[J]}),
a fundamental result in the structural analysis of linear systems
(\cite{Kouri23,Nicolau25,Sepitka25}),
reveals an important structural property of linear systems.
Under the complete equivalence transformations consisting of
\begin{itemize}
\item the algebraic equivalence transformations, including
  \begin{itemize}
  \item the state transformation,
  \item the input transformation, and
  \item the output transformation,
  \end{itemize}
\item the state feedback transformation, and
\item the output injection transformation,
\end{itemize}
a linear system can be transformed into a structured form composed of four subsystems:
an internal dynamic subsystem,
a strongly controllable subsystem,
a strongly observable subsystem, and
a strongly controllable and observable subsystem (also known as the prime subsystem).

The properties and applications of the Morse normal form have been explored
by many researchers, such as
(\cite{Kucera2019ECC[C],Kucera2020TAC[J],Patil2021,Chen2021,Zhou2024IJC[J]}),
among which the case of invertible linear systems
(\cite{Rospondek2017,Zhou2024IJC[J]})
is of particular interest to us.
On the one hand,
in the Morse normal form of an invertible system,
the strongly controllable subsystem
and the strongly observable subsystem do not appear.
On the other hand,
a system containing only an internal dynamic subsystem and
a prime subsystem has a regular relative degree
(\cite{Mueller2009[J],Zhou2023Auto[J],Guo}),
that is,
the coupling matrix of the system is nonsingular.
These two observations lead to the fact that,
under a series of complete equivalence transformations,
an invertible system can be transformed into
a system with a regular relative degree,
which is an important class of systems.

Among the complete equivalence transformations,
the regularity of the relative degree is invariant under
state transformations, input transformations,
and state feedback transformations,
while it can only be influenced by output transformations
and output injection transformations.
Summarizing the above analyses,
the Morse normal form reveals the existence of
output transformations and output injection transformations
that render an invertible system a regular relative degree.

Compared with existence results,
determination problems are often equally important.
The determination of the transformation matrices associated with the Morse normal form
has been extensively investigated in
\cite{Thorp1973,Jordan1977,vanDooren1979LAA[J],
Kitapci1984CDC[C],Dragan1987,Sannuti1987,Suda1994,Kucera20,Chen21}.
However, the existing approaches mainly focus on computing the transformation matrices
that bring a system into the Morse normal form,
and the corresponding computation procedures are usually recursive or implicit.
Moreover, these methods do not directly provide all admissible output injections
and output transformations that render an invertible system a regular relative degree.

Motivated by the above observations,
this paper investigates the determination of transformation matrices
for invertible linear time-invariant systems.
Different from the classical computation procedures for the Morse normal form,
which mainly aim at constructing the canonical structure,
this paper directly characterizes all admissible output injections
and output transformations that ensure the regularity of the relative degree.

The main contributions of this paper are summarized as follows.
\begin{itemize}

\item An explicit determination method is developed to compute all admissible
output injection and output transformation matrices that render
an invertible linear time-invariant system a regular relative degree.

\item The determination problem is reduced to solving a sequence of
linear matrix equations, which makes the proposed method explicit
and easy to implement.

\item Several properties of the block Hankel matrix are established
to support the theoretical development and the proof of the proposed method.

\end{itemize}

The remainder of this paper is organized as follows.
Section~2 presents the mathematical preliminaries and the problem formulation.
Section~3 provides the main results, describing the determination method
for both the output injection and the output transformation.
Section~4 presents a representative example to demonstrate
the application of the proposed computational method.
The proofs of the main results are given in Section~5.
Section~6 concludes the paper and discusses possible directions for future work.

\textbf{Notations:}
$\R$, $\mathbf C$, and $\mathbf Z$ denote the sets of real numbers,
complex numbers, and integers, respectively.
$\R^p$ denotes the set of $p$-dimensional real column vectors,
$\R^{p\times q}$ the set of $p\times q$ real matrices, and
$\R^{p\times q}_r$ the set of $p\times q$ real matrices with rank $r$.
$0_{p\times q}\in\R^{p\times q}$ denotes the $p\times q$ zero matrix,
$0_p\in\R^p$ the $p$-dimensional zero column vector, and
$I_p\in\R^{p\times p}$ the $p$-order identity matrix.
The symbol $0$ is used to denote either the number zero or
a zero matrix whose order is not specified.
For a matrix $X$,
$X^\T$ denotes its transpose;
$\rank(X)$ its rank;
$\det(X)$ its determinant if $X$ is square;
$\rref(X)$ its reduced row-echelon form;
and $\im(X)$ its column space.
By combining the symbols of transpose and image,
$\imr X$ is frequently used to denote the row space of $X$.
For submatrix representations,
$X[p,q],X[p,:],X[:,q],X[p_1\!:\!p_2,:]$ denote the $(p,q)$-entry,
the $p$-th row, the $q$-th column, and
the $p_1$-th to $p_2$-th rows of $X$, respectively.
In this paper, we frequently use the concept of an empty matrix.
For a matrix $X\in\R^{p\times q}$ or $X\in\R^{p\times q}_r$,
if $p=0$ or $q=0$, then $X$ is called an empty matrix.
The image of an empty matrix is the empty set, and conversely,
a matrix contained in the empty set is an empty matrix.
For a matrix summation $\Sigma_{i=p}^qX_i$,
if $q<p$, the summation is regarded as nonexistent.
A detailed remark on this unusual notation will be given
whenever it is encountered.
For two matrices $X$ and $Y$,
their difference is denoted by $X\backslash Y$,
their Kronecker product by $X\otimes Y$, and
$X\oplus Y$ denotes the block-diagonal matrix
$X\oplus Y=\begin{bmatrix}X&\\ &Y\end{bmatrix}$.
In a block matrix, the symbol ``$*$'' denotes a block that is not of concern.
For a sequence $a_1,a_2,\dots,a_n\in\R$,
$\nabla$ and $\nabla^2$ denote the first-order and second-order
backward differences, respectively, that is,
$\nabla a_i=a_i-a_{i-1}$,
$\nabla^2a_i=\nabla a_i-\nabla a_{i-1}$ for
$i=2,3,\dots,n$, and particularly, $\nabla a_1=a_1$.

\section{Preliminaries and problem statement}

\subsection{Output transformation and output injection}

Consider the square linear system
\begin{equation}\label{eqn: xdot=Ax+Bu, y=Cx}
\dot x=Ax+Bu,\ y=Cx.
\end{equation}
Here, $x\in\R^n$, $u,y\in\R^m$ are the state, input and output vectors;
and $A\in\R^{n\times n}$, $B\in\R^{n\times m}$, $C\in\R^{m\times n}$
are constant matrices.
The system \eqref{eqn: xdot=Ax+Bu, y=Cx} is called the system $(A,B,C)$ for short.

\begin{definition}[regular relative degree]\label{def:rd}
Denote $C=\begin{bmatrix}c_1^\T&c_2^\T&\cdots&c_m^\T\end{bmatrix}^\T$
in system \eqref{eqn: xdot=Ax+Bu, y=Cx},
where $c_i\in\R^{1\times n}$ for $i=1,2,\dots,m$.
The ordered set $\{r_1,r_2,\dots,r_m\}$, where $1\le r_i\le n$,
is said to be the relative degree of system
\eqref{eqn: xdot=Ax+Bu, y=Cx} if for $i=1,2,\dots,m$, there holds
$c_iA^jB=0$ for $j=0,1,\dots,r_i-2$, $c_iA^{r_i-1}B\ne0$.
We say the relative degree $\{r_1,r_2,\dots,r_m\}$ is regular if
$\det(\GAmma(A,B,C))\ne0$, where
\[
\GAmma(A,B,C)=
\begin{bmatrix}
(c_1A^{r_1-1}B)^\T&(c_2A^{r_2-1}B)^\T&\cdots&(c_mA^{r_m-1}B)^\T
\end{bmatrix}^\T.
\]
Here, $\GAmma(A,B,C)\in\R^{m\times m}$ is referred to
as the coupling matrix.
Conversely, the relative degree $\{r_1,r_2,\dots,r_m\}$
is said to be irregular if $\det(\GAmma(A,B,C))=0$.
\end{definition}

\begin{remark}
The importance of the regular relative degree lies in that
it is the sufficient and necessary condition for a linear system to
have an input-output normal form (\cite{Zhou2023Auto[J]}).
\end{remark}

For the system \eqref{eqn: xdot=Ax+Bu, y=Cx},
among the state transformation $\tilde x=T_xx$, $T_x\in\R^{n\times n}_n$,
the input transformation $\tilde u=T_uu$, $T_u\in\R^{m\times m}_m$ and
the output transformation $\tilde y=T_yy$, $T_y\in\R^{m\times m}_m$,
only the output transformation can influence
the regularity of the relative degree.
That is to say, for a system having an irregular relative degree,
only the output transformation may exist to make
the transformed system $(A,B,T_yC)$:
\[
\dot x=Ax+Bu,\ y=T_yCx,
\]
has a regular relative degree.
The sufficient and necessary condition about what kind of systems have
such a $T_y$ was proposed by \cite{Zhou2023Auto[J]} and \cite{Cao2025rrdref[J]}.

Finally, we introduce the output injection.
Output injection transformation refers to the introduction of
a proportional term for the output in the state equation,
i.e., $Ly$.
Here, $L\in\R^{n\times m}$ is usually called the Luenberger gain.
After output injection, the state equation of
the system \eqref{eqn: xdot=Ax+Bu, y=Cx} becomes
$\dot x=Ax+Bu+Ly=Ax+Bu+LCx=(A+LC)x+Bu$,
and thus the original system \eqref{eqn: xdot=Ax+Bu, y=Cx}
becomes the system $(A+LC,B,C)$:
\begin{equation*}
\dot x=(A+LC)x+Bu,\ y=Cx.
\end{equation*}
Performing the output injection and the output transformation in sequence,
we obtain the system $(A+LC,B,T_yC)$:
\begin{equation}\label{eqn: xdot=(A+LC)x+Bu, y=TyCx}
\dot x=(A+LC)x+Bu,\ y=T_yCx.
\end{equation}
The system \eqref{eqn: xdot=(A+LC)x+Bu, y=TyCx}
is what this paper aims to investigate.

\subsection{Invertible systems}

In this subsection,
we provide the definition of the invertible linear systems.

\begin{definition}\emph{(invertible system)}
Denote the transfer function of the system \eqref{eqn: xdot=Ax+Bu, y=Cx}
as $G(s)$ and the normal rank of the transfer function as
$\mathrm{nrank}(G(s))=\max_{s\in\mathbf C}\{\rank(G(s))\}$.
The system \eqref{eqn: xdot=Ax+Bu, y=Cx} is called invertible
if $\mathrm{nrank}(G(s))=m$.
\end{definition}

Then we need to introduce the concepts of
the invertibility matrices.

\begin{definition}
\emph{\cite{Zhoubook,Zhou2024IJC[J]}
(invertibility matrices and indices)}
The series of matrices
\[
\Qi^{[i]}(A,B,C)
=\begin{bmatrix}
&&&CB\\
&&CB&CAB\\
&\adots&\adots&\vdots\\
CB&CAB&\cdots&CA^{i-1}B
\end{bmatrix}
\in\R^{mi\times mi},
\quad i=1,2,\dots,n,
\]
are called the invertibility matrices of the system \eqref{eqn: xdot=Ax+Bu, y=Cx}.
The series of integers
$\faii^{[i]}=\faii^{[i]}(A,B,C)=\nabla\rank(\Qi^{[i]}(A,B,C))$
for $i=1,2,\dots,n$
are referred to as the invertibility indices of
the system \eqref{eqn: xdot=Ax+Bu, y=Cx}.
\end{definition}

For the sake of notational simplicity,
we sometimes directly use $\Qi^{[i]}$ to represent $\Qi^{[i]}(A,B,C)$ if no confusion may arise.
As for the invertibility indices,
there holds the following result.

\begin{lemma}\label{thm: invertible indices}
\emph{\cite{Sain1969TAC[J]}}
About the invertibility indices, there hold
\begin{enumerate}
\item The system \eqref{eqn: xdot=Ax+Bu, y=Cx} is invertible
if and only if $\faii^{[n]}=m$;\label{item:inv<=>fain=m}
\item $\faii^{[1]}\le\faii^{[2]}\le\cdots\le\faii^{[n]}$.\label{item:fai1<fain}
\end{enumerate}
\end{lemma}

\begin{remark}
\Cref{item:fai1<fain} of Lemma~\ref{thm: invertible indices} comes from a footnote of
\cite{Sain1969TAC[J]}.
A detailed proof can be found in \cite{Zhoubook}.
In this paper, it can also be proven by
\Cref{thm: ei=d'i} in the following.
\end{remark}

\subsection{Property and criterion}

In this subsection, at first, we introduce the rank property of
the invertibility matrices and the output controllability matrices
under some complete equivalence transformations.
Then, we introduce a concise criterion of the existence of
the output transformations making a linear system have a regular relative degree.
We introduce the output controllability matrices.

\begin{definition}
\emph{\cite{Zhoubook} (output controllability matrices)}
The series of matrices
\[
\Qoc^{[i]}(A,B,C)
=\begin{bmatrix}CB&CAB&\cdots&CA^{i-1}B\end{bmatrix}
\in\R^{m\times mi},
\quad i=1,2,\dots,n,
\]
are called the output controllability matrices of
the system \eqref{eqn: xdot=Ax+Bu, y=Cx}.
\end{definition}

Like $\Qi^{[i]}$, for the sake of notational simplicity,
we sometimes directly use $\Qi^{[i]}$ to represent $\Qi^{[i]}(A,B,C)$ if no confusion may arise.
Some rank properties are introduced below.

\begin{lemma}\label{thm: dont change Qi/oc}
\emph{\cite{Zhoubook}}
Consider the system \eqref{eqn: xdot=Ax+Bu, y=Cx}.
Suppose that $T_x\in\R^{n\times n}_n$, $T_u\in\R^{m\times m}_m$,
$T_y\in\R^{m\times m}_m$, $K\in\R^{m\times n}$, $L\in\R^{n\times m}$.
Then for $i=1,2,\dots,n$, there hold
\begin{gather}
\label{eqn: rank Qi A+BK}
\rank(\Qi^{[i]}(A+BK,B,C))=\rank(\Qi^{[i]}(A,B,C)),\\
\label{eqn: rank Qi A+LC}
\rank(\Qi^{[i]}(A+LC,B,C))=\rank(\Qi^{[i]}(A,B,C)),\\
\label{eqn: rank Qoc A+BK}
\rank(\Qoc^{[i]}(A+BK,B,C))=\rank(\Qoc^{[i]}(A,B,C)),\\
\label{eqn: rank Qi TxTuTy}
\rank(\Qi^{[i]}(T_xAT_x^{-1},T_xBT_u^{-1},T_yCT_x^{-1}))
=\rank(\Qi^{[i]}(A,B,C)),\\
\label{eqn: rank Qoc TxTuTy}
\rank(\Qoc^{[i]}(T_xAT_x^{-1},T_xBT_u^{-1},T_yCT_x^{-1}))
=\rank(\Qoc^{[i]}(A,B,C)).
\end{gather}
\end{lemma}

We introduce the criterion regarding the existence
of an output transformation such that the new system has a regular relative degree.

\begin{lemma}\label{thm: Zhou2023(8)(9)}
\emph{\cite{Zhou2023Auto[J]}}
Consider the system \eqref{eqn: xdot=Ax+Bu, y=Cx}.
There exists an output transformation matrix $T_y\in\R^{m\times m}_m$
such that the new system $(A,B,T_yC)$ has a regular relative
degree if and only if
\begin{equation}\label{eqn: Zhou2023(8)(9)}
\rank(\Qi^{[n]})
=\Sigma_{i=1}^n\rank(\Qoc^{[i]}),\
\rank(\Qoc^{[n]})=m.
\end{equation}
An alternative expression for \eqref{eqn: Zhou2023(8)(9)} is
\begin{equation}\label{eqn: Zhou2023(8)(9) n*}
\rank(\Qi^{[\nocstar]})
=\Sigma_{i=1}^{\nocstar}\rank(\Qoc^{[i]}),\
\rank(\Qoc^{[n]})=m,
\end{equation}
where
$\nocstar=\nocstar(A,B,C)=\min\big\{i\in\{1,2,\dots,n\}:\rank(\Qoc^{[i]}(A,B,C))=m\big\}$.
\end{lemma}

\subsection{Morse normal form and problem statement}

To analyze the system $(A+LC,B,T_yC)$,
we first recall the well-known Morse normal form (\cite{Morse1973SIAM[J]})
for invertible linear systems (\cite{Zhoubook,Zhou2024IJC[J]}).

\begin{lemma}\label{thm: Morse normal form}
Suppose that the linear system \eqref{eqn: xdot=Ax+Bu, y=Cx} is invertible.
Then there exists a quintuple of matrices $(T_x,T_u,T_y,K,L)$, where
$T_x\in\R^{n\times n}_n$, $T_u\in\R^{m\times m}_m$,
$T_y\in\R^{m\times m}_m$, $K\in\R^{m\times n}$, and $L\in\R^{n\times m}$,
such that the transformed system $(\tilde A,\tilde B,\tilde C)$ with
$\tilde A=T_x(A+BK+LC)T_x^{-1}$,
$\tilde B=T_xBT_u^{-1}$,
$\tilde C=T_yCT_x^{-1}$,
takes the following form:
\begin{equation}\label{eqn: Morse normal form's ABC}
\tilde A=A_0\oplus A_3,\
\tilde B=\begin{bmatrix}0_{m\times n_0}&B_3^\T\end{bmatrix}^\T,\
\tilde C=\begin{bmatrix}0_{m\times n_0}&C_3\end{bmatrix}.
\end{equation}
Here, $A_0\in\R^{n_0\times n_0}$ is a real Jordan matrix, and
$A_3\in\R^{n_3\times n_3}$,
$B_3\in\R^{n_3\times m}$,
$C_3\in\R^{m\times n_3}$
are given by
\begin{align*}
&
A_3=
\begin{bmatrix}
0&I_{q_1-1}\\ 0_1&0
\end{bmatrix}
\oplus\cdots\oplus
\begin{bmatrix}
0&I_{q_m-1}\\ 0_1&0
\end{bmatrix},
\\
&
B_3=
\begin{bmatrix}0_{q_1-1}\\ 1\end{bmatrix}
\oplus\cdots\oplus
\begin{bmatrix}0_{q_m-1}\\ 1\end{bmatrix},
\\
&
C_3=
\begin{bmatrix}1&0_{1\times(q_1-1)}\end{bmatrix}
\oplus\cdots\oplus
\begin{bmatrix}1&0_{1\times(q_m-1)}\end{bmatrix},
\end{align*}
where $q_1\ge q_2\ge\cdots\ge q_m\ge1$,
$n_0+n_3=n$, and $\Sigma_{i=1}^mq_i=n_3$.
\end{lemma}

The following proposition is a direct consequence of the Morse normal form.
It reveals the relationship between system invertibility and the existence of
output injection and transformation that enable a linear system to possess
a regular relative degree.
The proof is provided in Appendix.

\begin{prop}\label{thm: invertible and Morse form}
For the system \eqref{eqn: xdot=Ax+Bu, y=Cx},
there exist a Luenberger gain $L\in\R^{n\times m}$ and
an output transformation matrix $T_y\in\R^{m\times m}_m$
such that the new system $(A+LC,B,T_yC)$
has a regular relative degree
if and only if the system \eqref{eqn: xdot=Ax+Bu, y=Cx} is invertible.
\end{prop}

Although the Morse normal form is a fundamental result, the original work \cite{Morse1973SIAM[J]}
only provided a recursive and implicit solution for the Luenberger gain $L$.
This paper is dedicated to providing an explicit and computationally efficient
method to determine $L$ and $T_y$.
The formal problem statement is as follows.

\begin{problem}
For an invertible system \eqref{eqn: xdot=Ax+Bu, y=Cx},
characterize the set of all Luenberger gains $L\in\R^{n\times m}$
and output transformation matrices $T_y\in\R^{m\times m}_m$
such that the new system $(A+LC,B,T_yC)$
has a regular relative degree.
\end{problem}

We provide some further clarification on Problem~1.
Since the rows of $\Qi^{[n]}$ are composed of $\Qoc^{[i]}$ for $i=1,2,\dots,n$,
it is evident that
$\rank(\Qi^{[n]})\le\Sigma_{i=1}^n\rank(\Qoc^{[i]})$.
According to \eqref{eqn: rank Qi A+LC}, output injection
does not alter the rank of the invertibility matrices; thus,
$\rank(\Qi^{[n]}(A+LC,B,C))\le
\Sigma_{i=1}^n\rank(\Qoc^{[i]}(A+LC,B,C))$.
It follows from Lemma~\ref{thm: Zhou2023(8)(9)}
that to ensure the system $(A+LC,B,C)$ has a regular relative degree,
one must design a suitable $L$ to
minimize the ranks of $\Qoc^{[i]}(A+LC,B,C)$ such that
$\rank(\Qi^{[n]}(A+LC,B,C))=
\Sigma_{i=1}^n\rank(\Qoc^{[i]}(A+LC,B,C))$.

We note that Problem~1 is inherently nonlinear, as $L$
appears nonlinearly in the expressions for the output controllability matrices,
rendering the problem highly non-trivial.
The essence of our approach is to demonstrate that this nonlinear problem
can be transformed into a linear one.

\subsection{Block Hankel matrix and its reduced row-echelon form}
\label{sec: Hankel}

In this subsection, we establish several properties of block skew
lower-triangular Hankel matrices. These properties serve as a foundation
for characterizing the invertibility matrices in the subsequent subsection.

For $D_i \in \R^{m \times m}, i=1,2,\dots,n$,
denote the block skew lower-triangular Hankel matrix as
\[
D^{[i]}
=\begin{bmatrix}
&&&D_1\\
&&D_1&D_2\\
&\adots&\adots&\vdots\\
D_1&D_2&\cdots&D_n
\end{bmatrix}\in\R^{mi\times mi},
\quad i=1,2,\dots,n.
\]
Let $d_i = \rank(D^{[i]})$, $d_i' = \nabla d_i$, and
$d_i'' = \nabla^2 d_i$ for $i=1,2,\dots,n$.
The reduced row-echelon form of $D^{[n]}$ is partitioned as follows,
\begin{equation}\label{eqn:rrefDniswhat}
\rref(D^{[n]})
=\ \ \begin{NiceArray}{ccccWc{17pt}}
E_1&*&\cdots&*&e_1\\
&E_2&\cdots&*&e_2\\
&&\ddots&\vdots&\vdots\\
&&&E_n&e_n\\
0&0&\cdots&0&\\
m&m&\cdots&m&
\CodeAfter\SubMatrix[{1-1}{5-4}][right-xshift=5pt]
\tikz \draw (2-|1) -- (2-|5);
\tikz \draw (3-|1) -- (3-|5);
\tikz \draw (4-|1) -- (4-|5);
\tikz \draw (5-|1) -- (5-|5);
\tikz \draw (1-|2) -- (6-|2);
\tikz \draw (1-|3) -- (6-|3);
\tikz \draw (1-|4) -- (6-|4);
\end{NiceArray}.
\end{equation}
Here, $E_i \in \R^{e_i \times m}_{e_i}$ represents a matrix of full row rank,
where $e_i \ge 0$ for $i=1,2,\dots,n$.
The following lemma summarizes its properties; the proof is provided in Appendix.

\begin{lemma}\label{thm: ei=d'i}
The following properties hold:
\begin{enumerate}
\item \label{thm: Drref=E1~i}
The reduced row-echelon form of $D^{[i]}$ is given by
\begin{equation}\label{eqn: Dirref=E1~i}
\rref(D^{[i]})=
\begin{bmatrix}
E_1&*&\cdots&*\\
&E_2&\cdots&*\\
&&\ddots&\vdots\\
&&&E_i\\
\multicolumn{4}{c}{0}
\end{bmatrix},
\quad i=1,2,\dots,n;
\end{equation}
\item \label{item:ei=d'i}
$e_i = d'_i$ for $i=1,2,\dots,n$;
\item \label{item:E1<En}
$\im(E_1^\T) \subseteq \im(E_2^\T) \subseteq \cdots \subseteq \im(E_n^\T)$;
\item \label{item:d1<dn-dn-1}
$d_1 \le d_2 - d_1 \le d_3 - d_2 \le \cdots \le d_n - d_{n-1}$.
\end{enumerate}
\end{lemma}

\begin{remark}
By identifying $D_1, D_2, \dots, D_n$ with $CB, CAB, \dots, CA^{n-1}B$,
\Cref{item:fai1<fain} of Lemma~\ref{thm: invertible indices} follows
directly from \Cref{item:d1<dn-dn-1} of \Cref{thm: ei=d'i}.
\end{remark}

Next, we introduce a specific structure of the row-echelon form.
The following lemma, whose proof is provided in Appendix,
is a direct consequence of Items~\ref{thm: Drref=E1~i} and \ref{item:E1<En} of \Cref{thm: ei=d'i}.

\begin{lemma}\label{thm: QD=F}
Let $i$ be any integer such that $1 \le i \le n$.
For each $j=1,2,\dots,i$,
let $f_j$ be an integer satisfying $0 \le f_j \le d''_j$,
and let $F_j \in \R^{f_j \times m}_{f_j}$ be any matrix
satisfying $\im(F_j^\T) \setminus \{0_m\} \subseteq \im(E_j^\T) \setminus \im(E_{j-1}^\T)$,
where $E_0 = \{0_m\}$.
Then, the following results hold:
\begin{enumerate}
\item\label{thm:existQ}
There exists a matrix $Q \in \R^{f \times mi}$ with $f = \Sigma_{j=1}^i f_j$
such that
\[
QD^{[i]}=\begin{bmatrix}
F_1&*&\cdots&*\\
&F_2&\cdots&*\\
&&\ddots&\vdots\\
&&&F_i
\end{bmatrix};
\]
\item\label{thm:imF=imE}
If $f_j = d''_j$ for $j=1,2,\dots,i$, then
$\im\begin{bmatrix}F_1^\T & F_2^\T & \cdots & F_i^\T\end{bmatrix}
= \im(E_i^\T)$;
\item\label{thm:Q1rank}
By partitioning $Q$ as $Q = \begin{bmatrix} * & Q_1 \end{bmatrix}$
where $Q_1 \in \R^{f \times m}$,
the submatrix $Q_1$ has full row rank.
\end{enumerate}
\end{lemma}

\section{Main results}

In this section, we state the determination method for
the Luenberger gain $L$ and the output transformation matrix $T_y$.
The proof for the determination method is provided in the next section.

For an invertible system, it follows from
Lemma~\ref{thm: invertible indices} that
$\faii^{[1]}\le\faii^{[2]}\le\cdots\le\faii^{[n]}=m$.
Then, an important number $\nsharp$ can be defined as
$\nsharp=\nsharp(A,B,C)=\min\big\{i\in\{1,2,\dots,n\}:\faii^{[i]}(A,B,C)=m\big\}$.

\begin{remark}\label{rmk: n* and nsharp}
The relation between $\nocstar$ and the newly defined $\nsharp$
can be summarized as follows.
Their proofs are in Appendix.
\begin{enumerate}
\item If the system \eqref{eqn: xdot=Ax+Bu, y=Cx} satisfies
\eqref{eqn: Zhou2023(8)(9)}, then $\nsharp=\nocstar$;
\item if the system \eqref{eqn: xdot=Ax+Bu, y=Cx} is invertible
and the $L\in\R^{n\times m}$ is chosen such that
there exists a $T_y\in\R^{m\times m}_m$
making the system $(A+LC,B,T_yC)$
have a regular relative degree,
then $\nocstar(A+LC,B,C)=\nsharp(A,B,C)$;
\item $\nsharp\ge \nocstar$;
\item $\nocstar=2$ if $\nsharp=2$ (see Lemma~\ref{thm: if nsharp=2 then n*=2}).
\end{enumerate}
\end{remark}

Firstly, we state the results for the cases $\nsharp=1,2$
in which the Luenberger gain $L$ doesn't need to be figured out.
Then, we introduce the cases $\nsharp\ge3$,
which is the main thread of this paper.
The case $\nsharp=1$ is quite simple and is stated in the following.

\begin{prop}
If the system \eqref{eqn: xdot=Ax+Bu, y=Cx} is invertible and
$\nsharp=1$, then this system has a regular relative degree.
\end{prop}

\begin{proof}
The condition $\nsharp=1$ means that $\Qi^{[1]}=CB$ is nonsingular,
implying that the coupling matrix $\GAmma(A,B,C)=CB$ is already nonsingular.
Hence, this system has a regular relative degree.
This finishes the proof.
\end{proof}

\subsection{Case $\nsharp=2$}

In this subsection, we investigate the case where $\nsharp=2$.
The following lemma, which also appears as
Item~4 of Remark~\ref{rmk: n* and nsharp},
is required for the subsequent analysis, and its proof is provided in Appendix.

\begin{lemma}\label{thm: if nsharp=2 then n*=2}
For the system \eqref{eqn: xdot=Ax+Bu, y=Cx},
if $\nsharp=2$, then $\nocstar=2$.
\end{lemma}

The following result reveals that,
even without output injection,
an invertible system with $\nsharp=2$ already satisfies
the condition \eqref{eqn: Zhou2023(8)(9) n*}.

\begin{prop}\label{thm:ifnsharp=2}
If the system \eqref{eqn: xdot=Ax+Bu, y=Cx} is invertible and $\nsharp=2$,
then the system satisfies
\eqref{eqn: Zhou2023(8)(9) n*}.
\end{prop}

\begin{proof}
The condition $\nsharp=2$ implies
\begin{equation}\label{eqn: faii2=m}
\faii^{[2]}=\rank\begin{bmatrix}&CB\\ CB&CAB\end{bmatrix}-\rank(CB)=m.
\end{equation}
By Lemma~\ref{thm: if nsharp=2 then n*=2},
it follows that $\nocstar=2$, which means that
\begin{equation}\label{eqn: rank Qoc2=m}
\rank\begin{bmatrix}CB&CAB\end{bmatrix}=m.
\end{equation}
Combining \eqref{eqn: faii2=m}
and \eqref{eqn: rank Qoc2=m} yields
\begin{equation}\label{eqn: rank Qi2=rank Qoc1+rank Qoc2}
\rank\begin{bmatrix}&CB\\ CB&CAB\end{bmatrix}=\rank(CB)
+\rank\begin{bmatrix}CB&CAB\end{bmatrix}.
\end{equation}
The conjunction of \eqref{eqn: rank Qoc2=m} and
\eqref{eqn: rank Qi2=rank Qoc1+rank Qoc2}
indicates that the system \eqref{eqn: xdot=Ax+Bu, y=Cx} satisfies
\eqref{eqn: Zhou2023(8)(9) n*}.
This completes the proof.
\end{proof}

We consider the system under both output injection and transformation, namely, the system $(A+LC,B,T_yC)$.

\begin{theorem}
\label{thm: (A,B,TyC) and (A+LC,B,TyC), Ty the same}
Consider the invertible system \eqref{eqn: xdot=Ax+Bu, y=Cx}
with $\nsharp=2$.
Let $T_y\in\R^{m\times m}_m$ be an output transformation matrix.
Then for any Luenberger gain $L\in\R^{n\times m}$,
the system $(A,B,T_yC)$ has a regular relative degree
if and only if
the system $(A+LC,B,T_yC)$ has a regular relative degree.
\end{theorem}

\begin{proof}
We focus on proving the necessity, as the sufficiency follows by a similar argument and is thus omitted for brevity.

According to \Cref{thm: if nsharp=2 then n*=2}, the condition $\nsharp=2$ implies $\nocstar=2$, which yields
\[
\rank(T_yCB)<m, \ \rank\begin{bmatrix}T_yCB & T_yCAB\end{bmatrix}=m.
\]
Consequently, the relative degree of the system $(A,B,T_yC)$ must be either $\{1,\dots,1,2,\dots,2\}$ or $\{2,\dots,2\}$.
Suppose the relative degree is
\[
\{\underbrace{1,\dots,1}_{n_1\,\mathrm{times}}, \underbrace{2,\dots,2}_{n_2\,\mathrm{times}}\},
\text{ where $n_1$ may be $0$.}
\]
Without loss of generality, let $T_y$ be partitioned as $T_y=\begin{bmatrix}T_1\\ T_2\end{bmatrix}$, where $T_1\in\R^{n_1\times m}$ and $T_2\in\R^{n_2\times m}$.
According to the definition of the regular relative degree (\Cref{def:rd}), it holds that $T_2CB=0$, and the coupling matrix
$\GAmma(A,B,T_yC)=\begin{bmatrix}T_1CB\\ T_2CAB\end{bmatrix}$
is nonsingular. By noting that $T_yCB = \begin{bmatrix}T_1CB \\ 0\end{bmatrix}$ and partitioning $T_2CLT_y^{-1}$ as $\begin{bmatrix}T_3 & *\end{bmatrix}$ where $T_3\in\R^{n_2\times n_1}$, we obtain
\begin{align*}
T_2C(A+LC)B
&= T_2CAB + T_2CLCB
= T_2CAB + T_2CLT_y^{-1}T_yCB \\
&= T_2CAB + \begin{bmatrix}T_3 & *\end{bmatrix} \begin{bmatrix}T_1CB \\ 0\end{bmatrix}
= T_2CAB + T_3T_1CB.
\end{align*}
Since $\begin{bmatrix}T_1CB \\ T_2CAB\end{bmatrix}$ is nonsingular, it follows that each row of $T_2CAB + T_3T_1CB$ is non-zero.
Thus, the coupling matrix is
\[
\GAmma(A+LC,B,T_yC) = \begin{bmatrix}T_1CB \\ T_2C(A+LC)B\end{bmatrix} = \begin{bmatrix}I_{n_1} & \\ T_3 & I_{n_2}\end{bmatrix} \begin{bmatrix}T_1CB \\ T_2CAB\end{bmatrix},
\]
and it is also nonsingular. This implies that the system $(A+LC,B,T_yC)$ also possesses a regular relative degree, which completes the proof.
\end{proof}

Theorem~\ref{thm: (A,B,TyC) and (A+LC,B,TyC), Ty the same} demonstrates that, for the case $\nsharp=2$,
the selection of the output transformation
matrix $T_y$ is decoupled from
the choice of the Luenberger gain $L$.
The following proposition, providing an equivalent description of
this result, follows immediately.

\begin{prop}\label{thm: nsharp=2 Ty L no}
Consider the invertible system \eqref{eqn: xdot=Ax+Bu, y=Cx}
where $\nsharp=2$.
Let $T_y\in\R^{m\times m}_m$ be an output transformation matrix.
Then for any Luenberger gains $L_1,L_2\in\R^{n\times m}$,
the system $(A+L_1C,B,T_yC)$ has a regular relative degree
if and only if
the system $(A+L_2C,B,T_yC)$ has a regular relative degree.
\end{prop}

\subsection{Cases $\nsharp \ge 3$}

In this subsection, we investigate the cases where $\nsharp \ge 3$.
We first partition the reduced row-echelon form of the invertibility matrix $\Qi^{[\nsharp]}$ as
\[
\rref(\Qi^{[\nsharp]})=
\ \ \begin{NiceArray}{cccc}
E_1&*&\cdots&*\\
&E_2&\cdots&*\\
&&\ddots&\vdots\\
&&&E_{\nsharp}\\
0&0&\cdots&0\\
m&m&\cdots&m
\CodeAfter\SubMatrix[{1-1}{5-4}]
\tikz \draw (2-|1) -- (2-|5);
\tikz \draw (3-|1) -- (3-|5);
\tikz \draw (4-|1) -- (4-|5);
\tikz \draw (5-|1) -- (5-|5);
\tikz \draw (1-|2) -- (6-|2);
\tikz \draw (1-|3) -- (6-|3);
\tikz \draw (1-|4) -- (6-|4);
\end{NiceArray}\ \ .
\]
Here, $E_i$ for $i=1,2,\dots,\nsharp$ each has $m$ columns
and possesses full row rank.
Note that the invertibility matrix $\Qi^{[\nsharp]}$
is a block skew lower-triangular Hankel matrix.
Consequently, based on the results from Subsection \ref{sec: Hankel},
the following properties hold:
\begin{enumerate}
\item According to \Cref{item:ei=d'i} of \Cref{thm: ei=d'i},
one has $E_i \in \R^{\faii^{[i]} \times m}_{\faii^{[i]}}$
for $i=1,2,\dots,\nsharp$.
In particular, $E_{\nsharp} \in \R^{m \times m}_m$ is nonsingular.
\item From \Cref{item:E1<En} of \Cref{thm: ei=d'i}, it follows that
$\im(E_1^\T) \subseteq \im(E_2^\T) \subseteq \cdots \subseteq \im(E_{\nsharp}^\T)$.
\item
Define
\begin{equation}\label{eqn: define psii}
\psi_i = \nabla\faii^{[i]},\quad i=1,2,\dots,\nsharp.
\end{equation}
By \Cref{item:fai1<fain} of Lemma~\ref{thm: invertible indices},
it follows that $\psi_i \ge 0$ for $i=1,2,\dots,\nsharp$.
By \Cref{thm:existQ} of Lemma~\ref{thm: QD=F},
given $F_i \in \R^{\psi_i \times m}_{\psi_i}$ satisfying $\im(F_i^\T) \setminus \{0_m\} \subseteq \im(E_i^\T) \setminus \im(E_{i-1}^\T)$ (with $E_0 = \{0_m\}$),
there exists a matrix $R \in \R^{m \times m\nsharp}$
such that
\begin{equation}\label{eqn: RQi=G1~n}
R\Qi^{[\nsharp]}=\begin{bmatrix}
F_1&*&\cdots&*\\
&F_2&\cdots&*\\
&&\ddots&\vdots\\
&&&F_{\nsharp}
\end{bmatrix}=F \in \R^{m \times m\nsharp}.
\end{equation}
Note that $R$ and $F$ in \eqref{eqn: RQi=G1~n} are not unique.
\item From \Cref{thm:imF=imE} of Lemma~\ref{thm: QD=F},
the matrix $\begin{bmatrix} F_1^\T & F_2^\T & \cdots & F_{\nsharp}^\T \end{bmatrix}$
is nonsingular.
\item
Partition $R$ as
\begin{equation}\label{eqn: partition R}
R = \begin{bmatrix} R_{\nsharp} & \cdots & R_2 & R_1 \end{bmatrix},
\end{equation}
where $R_i \in \R^{m \times m}$ for $i=1,2,\dots,\nsharp$.
By \Cref{thm:Q1rank} of \Cref{thm: QD=F}, $R_1$ is nonsingular.
\end{enumerate}
For $i=2,3,\dots,n$, define the following sequence of matrices:
\begin{equation*}
\Pio^{[i]}(L)
= \begin{bmatrix}
I_m & & & & \\
-CL & I_m & & & \\
-CAL & -CL & I_m & & \\
\vdots & \ddots & \ddots & \ddots & \\
-CA^{i-2}L & \cdots & -CAL & -CL & I_m \\
\end{bmatrix}
\in \R^{mi \times mi}.
\end{equation*}
Furthermore, let
\begin{equation}\label{eqn: define S}
S(L) = R \Pio^{[\nsharp]} (I_{\nsharp} \otimes R_1^{-1})
\in \R^{m \times m\nsharp},
\end{equation}
which is partitioned as
\begin{equation}\label{eqn: partition S}
S(L) = \begin{bmatrix}
S_{\nsharp}(L) & \cdots & S_3(L) & S_2(L) & I_m
\end{bmatrix},
\end{equation}
where $S_i(L) \in \R^{m \times m}$ for $i=2,3,\dots,\nsharp$.
For brevity, we may denote $S_i(L)$ as $S_i$.
Using the notation $\psi_{i \sim j} = \Sigma_{k=i}^j \psi_k$ (for $i \le j$),
we define:
\begin{align*}
\PHi_{i \sim j}
&= \begin{bmatrix}
0_{\psi_{i \sim j} \times \psi_{1 \sim (i-1)}}
& I_{\psi_{i \sim j}}
& 0_{\psi_{i \sim j} \times \psi_{(j+1) \sim \nsharp}}
\end{bmatrix} \in \R^{\psi_{i \sim j} \times m},\\
\PSi_{i}
&= \begin{bmatrix}
0_{\psi_{1 \sim (i-1)} \times \psi_i} \\
I_{\psi_i} \\
0_{\psi_{(i+1) \sim \nsharp} \times \psi_i}
\end{bmatrix} \in \R^{m \times \psi_{i}}.
\end{align*}

The main result of this paper---the determination method of
the Luenberger gain $L$ for $\nsharp \ge 3$---is stated below.
Its proof is comprehensive and is deferred to the next section.

\begin{theorem}\label{thm: theorem}
Consider the invertible system \eqref{eqn: xdot=Ax+Bu, y=Cx}
where $\nsharp \ge 3$.
Then there exists an output transformation matrix $T_y \in \R^{m \times m}_m$
such that the
new system $(A+LC, B, T_yC)$
has a regular relative degree if and only if
the Luenberger gain $L \in \R^{n \times m}$ satisfies
the following linear matrix equations:
\begin{equation}\label{eqn: main result}
\PHi_{(i+j)\sim \nsharp} S_i(L) \PSi_{j} = 0,\quad
j=1,2,\dots,\nsharp-i,\
i=2,3,\dots,\nsharp-1.
\end{equation}
\end{theorem}

The determination method of the output transformation matrix $T_y$
is introduced below, following the results in \cite{Cao2025rrdref[J]}.
The corresponding proof is also provided in the next section.

\begin{theorem}\label{thm: theorem solve Ty}
For an invertible system \eqref{eqn: xdot=Ax+Bu, y=Cx} where $\nsharp \ge 3$,
if the Luenberger gain $L$ satisfies \eqref{eqn: main result},
then the new system $(A+LC, B, T_yC)$
has a regular relative degree if and only if
the output transformation matrix $T_y \in \R^{m \times m}_m$ satisfies
\begin{equation}\label{eqn: Ty=PUR1}
T_y = PUR_1,
\end{equation}
where $P \in \R^{m \times m}$ is an arbitrary permutation matrix
and $U \in \R^{m \times m}$ takes the form
\begin{equation*}
U = \begin{bmatrix}
U_1 & * & \cdots & * \\
& U_2 & \cdots & * \\
& & \ddots & \vdots \\
& & & U_{\nsharp}
\end{bmatrix},
\text{ with $U_i \in \R^{\psi_i \times \psi_i}_{\psi_i}$ for $i=1,2,\dots,\nsharp$.}
\end{equation*}
\end{theorem}

From Theorems~\ref{thm: theorem} and~\ref{thm: theorem solve Ty},
noting that $R_1$ is independent of $L$,
we obtain the following result, which demonstrates that
the choices of $L$ and $T_y$ are decoupled.

\begin{prop}\label{thm: nsharp=3 Ty L no}
Consider the invertible system \eqref{eqn: xdot=Ax+Bu, y=Cx}
where $\nsharp \ge 3$.
Let $T_y \in \R^{m \times m}_m$ be an output transformation matrix.
Then for any Luenberger gains $L_1, L_2 \in \R^{n \times m}$
satisfying \eqref{eqn: main result},
the system $(A+L_1C, B, T_yC)$ has a regular relative degree
if and only if
the system $(A+L_2C, B, T_yC)$ has a regular relative degree.
\end{prop}

\begin{remark}
In contrast to Proposition~\ref{thm: nsharp=2 Ty L no},
the Luenberger gains $L_1, L_2$
in Proposition~\ref{thm: nsharp=3 Ty L no} must satisfy \eqref{eqn: main result}.
\end{remark}

The following result reveals that if the system $(A+LC, B, T_yC)$
possesses a regular relative degree, its relative degree is
uniquely determined by $\psi_i$ for $i=1,2,\dots,\nsharp$,
and thus by the invertibility indices $\faii^{[i]}$.
The proof is presented in the next section.

\begin{prop}\label{thm: relative degree of (A+LC,B,TyC)}
For an invertible system \eqref{eqn: xdot=Ax+Bu, y=Cx} with $\nsharp\ge 3$,
if $L \in \R^{n \times m}$ and $T_y \in \R^{m \times m}_m$
satisfy \eqref{eqn: main result} and \eqref{eqn: Ty=PUR1}, respectively,
then the relative degree of the system $(A+LC, B, T_yC)$ is unique
up to the rearrangement of the following set:
\begin{equation}\label{eqn: relative degree psi1...psinsharp}
\{\underbrace{1,\dots,1}_{\psi_1\,\mathrm{times}},
\underbrace{2,\dots,2}_{\psi_2\,\mathrm{times}},\dots,
\underbrace{\nsharp,\dots,\nsharp}_{\psi_{\nsharp}\,\mathrm{times}}\},
\end{equation}
where $i \in \{1,2,\dots,\nsharp\}$ is absent if $\psi_i=0$.
\end{prop}

It is worth noting that Exercise 4.32 of \cite{Zhoubook}
can be viewed as a corollary of Proposition~\ref{thm: relative degree of (A+LC,B,TyC)}.
The determination method for the Luenberger gain $L$ and
the output transformation matrix $T_y$ for $\nsharp \ge 3$
is summarized in Algorithm~1.

\noindent\rule{\linewidth}{1pt}
\noindent\textbf{Algorithm 1}: Determine all possible $L$ and $T_y$

\begin{algorithmic}

\State \textbf{Step 1:} Determine $\nsharp$:
\For{$i=1$ to $n$}
  \If{$\faii^{[i]}=m$}
    \State $\nsharp \gets i$
    \State \textbf{break}
  \EndIf
\EndFor

\If{$\nsharp=1$ or $2$}
  \State \Return \Comment{Terminate: Trivial case encountered}
\Else
  \State \textbf{Step 2} \Comment{Naturally proceeds to the next step}
\EndIf

\State \textbf{Step 2:} Calculate $\rref(\Qi^{[\nsharp]})$ and determine
$F_i$, $i=1,2,\dots,\nsharp$

\State \textbf{Step 3:} Select the rows of $\rref(\Qi^{[\nsharp]})$ to obtain $F$

\State \textbf{Step 4:} Solve the linear matrix equation
$R\Qi^{[\nsharp]}=F$ to obtain $R$ and $R_1$

\State \textbf{Step 5:} Determine the general form of all possible $T_y$
according to~\eqref{eqn: Ty=PUR1}

\State \textbf{Step 6:} Compute
$\PHi_{(i+j)\sim \nsharp} S_i \PSi_j$
in~\eqref{eqn: main result} by
$S = R\Pio^{[\nsharp]}(I_{\nsharp}\otimes R_1^{-1})$

\State \textbf{Step 7:} Solve the linear matrix equation
\eqref{eqn: main result} to obtain $L$

\end{algorithmic}

\vspace{-0.5em}\noindent\rule{\linewidth}{1pt}\vspace{-1em}

\section{Numerical Example}

Consider the system \eqref{eqn: xdot=Ax+Bu, y=Cx} from \cite{Fomichev2017DE[J]} with $n=9$ and $m=3$.
The coefficient matrices are given by
\begin{align*}
&\left[\begin{array}{c|c}
	A&B\\ \hline
	C&
\end{array}\right]
=\left[\begin{array}{ccccccccc|ccc}
0&1&0&0&0&0&0&0&0 & 0&0&1\\
0&0&1&0&0&0&0&0&0 & 0&0&0\\
1&0&1&0&0&0&0&1&0 & 1&0&0\\
0&0&0&0&1&0&0&0&0 & 0&0&0\\
0&0&0&0&0&1&0&0&0 & 0&0&1\\
0&0&0&0&1&0&1&1&1 & 0&1&0\\
0&0&0&0&0&0&0&1&0 & 0&0&0\\
0&0&0&0&0&0&0&0&1 & 0&0&0\\
0&1&0&0&0&1&0&0&0 & 0&0&1\\
\hline
1&0&0&0&0&0&0&0&0 & &&\\
0&0&0&1&0&0&0&0&0 & &&\\
0&0&0&0&0&0&1&0&0 & &&\\
\end{array}\right].
\end{align*}
The reduced row-echelon form of the output controllability matrix $\Qoc^{[2]}$ is
\[
\rref(\Qoc^{[2]})
=\rref\begin{bmatrix}CB&CAB\end{bmatrix}
=\left[\begin{array}{ccc|ccc}
0&0&1 & 0&0&0\\
0&0&0 & 0&0&1\\
0&0&0 & 0&0&0
\end{array}\right].
\]
According to \Cref{thm: rrd iff detG not 0} of \Cref{thm:rrdrefresult},
there is no output transformation matrix $T_y$ that enables
the transformed system $(A,B,T_yC)$ to possess a regular relative degree.
\\
\textbf{Step 1: Determine $\nsharp$}
\\
The ranks of the invertibility matrices are
\[\rank(\Qi^{[1]})=1,
\rank(\Qi^{[2]})=2,
\rank(\Qi^{[3]})=4,
\rank(\Qi^{[4]})=6,
\rank(\Qi^{[5]})=9.\]
Consequently, the invertibility indices are
\[\faii^{[1]}=1,\
\faii^{[2]}=1,\
\faii^{[3]}=2,\
\faii^{[4]}=2,\
\faii^{[5]}=3=m.\]
Based on Item~\ref{item:inv<=>fain=m} of Lemma~\ref{thm: invertible indices},
the system is invertible with $\nsharp=5$.
\\
\textbf{Step 2: Obtain Reduced Row-echelon Form}
\\
The reduced row-echelon form of the invertibility matrix
$\Qi^{[\nsharp]}=\Qi^{[5]}$ is
\begin{equation}\label{eqn:ex rrefQisharp}
\rref(\Qi^{[5]})=
\begin{bNiceArray}{>{\strut}ccccccccccccccc}%
[create-extra-nodes,extra-margin=2pt,code-after = {
	\tikz{
		\draw[name suffix = -large]
			(1-1.south west) -| (2-4.south west) -| (4-7.south west)
			-| (6-10.south west) -| (9-13.south west) -- (9-15.south east);
		\foreach \row in {1,3,7} {
			\node[anchor=west, xshift=5pt] at (\row-15.east) {$\checkmark$};
		}
	}
}]
0&0&1&0&0&0&1&0&0&1&0&0&0&0&0\\
0&0&0&0&0&1&0&0&0&1&0&0&0&0&0\\
0&0&0&0&0&0&0&1&0&-1&0&0&0&0&0\\
0&0&0&0&0&0&0&0&1&0&0&0&0&0&0\\
0&0&0&0&0&0&0&0&0&0&1&0&0&0&0\\
0&0&0&0&0&0&0&0&0&0&0&1&0&0&0\\
0&0&0&0&0&0&0&0&0&0&0&0&1&0&0\\
0&0&0&0&0&0&0&0&0&0&0&0&0&1&0\\
0&0&0&0&0&0&0&0&0&0&0&0&0&0&1\\
\multicolumn{15}{c}{0_{6\times 15}}
\end{bNiceArray}\ \ \ .
\end{equation}
We determine:
$E_1=E_2=\begin{bmatrix}0&0&1\end{bmatrix}$,
$E_3=E_4=\begin{bmatrix}0&1&0\\ 0&0&1\end{bmatrix}$, and
$E_5=I_3$.
\\
\textbf{Step 3: Calculation of $F$}
\\
Note that
$\psi_1=1$,
$\psi_2=0$,
$\psi_3=1$,
$\psi_4=0$,
$\psi_5=1$.
By selecting the 1st, 3rd, and 7th rows of $\rref(\Qi^{[5]})$ (marked with ``$\checkmark$'' in \eqref{eqn:ex rrefQisharp}),
there exists a matrix $R$ such that $R\Qi^{[5]}=F$, where
\[
F=
\begin{bNiceArray}{>{\strut}ccccccccccccccc}%
[create-extra-nodes,extra-margin=2pt,
code-after = {\tikz\draw[name suffix = -large]
(1-1.south west) -| (2-7.south west) -| (3-13.south west)
|- (3-15.south east);}]
0&0&1&0&0&0&1&0&0&1&0&0&0&0&0\\
0&0&0&0&0&0&0&1&0&-1&0&0&0&0&0\\
0&0&0&0&0&0&0&0&0&0&0&0&1&0&0
\end{bNiceArray}.
\]
Note that $F_1=\begin{bmatrix}0&0&1\end{bmatrix}$,
$F_3=\begin{bmatrix}0&1&0\end{bmatrix}$,
$F_5=\begin{bmatrix}1&0&0\end{bmatrix}$,
and $F_2$, $F_4$ are empty.
\\
\textbf{Step 4: Solving for Matrix $R$}
\\
Solving $R\Qi^{[5]}=F$ for $R$ yields
\begin{align*}
\begin{array}{l}
R=\left[\begin{array}{ccc|ccc|ccc|ccc|ccc}
-\frac23&0&0 & -1&-\frac23&0 & 0&0&-\frac23 & 0&1&0 & 1&0&-1\\
0&0&0 & 0&0&0 & 0&-1&0 & -1&-1&-1 & 0&1&1\\
0&0&0 & 0&0&0 & 0&0&0 & 0&-1&0 & 0&0&1
\end{array}\right]\\[15pt]
\hspace{18mm}R_5\hspace{16mm}R_4\hspace{15mm}R_3
\hspace{20mm}R_2\hspace{16mm}R_1
\end{array}\!\!\!.
\end{align*}
Note that $R_1$ is nonsingular, consistent with \Cref{thm:Q1rank} of \Cref{thm: QD=F}.
\\
\textbf{Step 5: Determining the Output Transformation Matrix $T_y$}
\\
From Theorem~\ref{thm: theorem solve Ty}, the solutions for $T_y$ are
\begin{equation*}
\begin{aligned}
T_y&=PUR_1
=P\begin{bmatrix}u_1&v_1&v_2\\ &u_3&v_3\\ &&u_5\end{bmatrix}
\begin{bmatrix}1&0&-1\\ 0&1&1\\ 0&0&1\end{bmatrix}
=P\begin{bmatrix}
u_1&v_1&-u_1+v_1+v_2\\
0&u_3&u_3+v_3\\
0&0&u_5
\end{bmatrix},
\end{aligned}
\end{equation*}
where $u_1,u_3,u_5\in\R$, $u_1,u_3,u_5\ne0$
and $v_1,v_2,v_3\in\R$ are arbitrary numbers.
This can be compactly written as
$T_y=P\begin{bmatrix}t_1&*&*\\ &t_2&*\\ &&t_3\end{bmatrix}$ with $t_i\in\R$, $t_i \neq 0$ for $i=1,2,3$.
Consistent with Proposition~\ref{thm: nsharp=3 Ty L no}, $T_y$ is independent of $L$.
\\
\textbf{Step 6: Construction of the Luenberger Gain Equations}
\\
Recalling the definition of $S(L)$ in \eqref{eqn: define S} and its partition in \eqref{eqn: partition S}, we have
\[
S(L) = R\Pio^{[5]}[I_5\otimes(R_1)^{-1}] = \begin{bmatrix}S_5&S_4&S_3&S_2&I_m\end{bmatrix},
\]
which expands to
\begin{align*}
&\begin{bmatrix}S_5&S_4&S_3&S_2&I_m\end{bmatrix}
=\begin{bmatrix}R_5&R_4&R_3&R_2&R_1\end{bmatrix}\\
&\times
\begin{bmatrix}
I_3&&&&\\
-CL&I_3&&&\\
-CAL&-CL&I_3&&\\
-CA^2L&-CAL&-CL&I_3&\\
-CA^3L&-CA^2L&-CAL&-CL&I_3\\
\end{bmatrix}
\begin{bmatrix}
R_1^{-1}&&&&\\
&R_1^{-1}&&&\\
&&R_1^{-1}&&\\
&&&R_1^{-1}&\\
&&&&R_1^{-1}
\end{bmatrix}.
\end{align*}
Specifically,
\begin{align*}
S_2
&
= (R_2 - R_1CL)R_1^{-1},
\\
S_3
&
= (R_3 - R_2CL - R_1CAL)R_1^{-1},
\\
S_4
&
= (R_4 - R_3CL - R_2CAL - R_1CA^2L)R_1^{-1}.
\end{align*}
Now we are ready to calculate the Luenberger gain $L=(l_{ij})_{9\times3}$.
It follows from \Cref{thm: theorem} that $L$ satisfies
\begin{equation}\label{eqn: example n=9 m=3 solve L}
\begin{gathered}
\PHi_{3\sim5}S_2\PSi_1=0,\
\PHi_{4\sim5}S_2\PSi_2=0,\
\PHi_5S_2\PSi_3=0,\\
\PHi_{4\sim5}S_3\PSi_1=0,\
\PHi_5S_3\PSi_2=0,\
\PHi_5S_4\PSi_1=0.
\end{gathered}
\end{equation}
By noting that $\psi_2=0$,
it follows that $\PSi_2$ is empty,
thus $\PHi_{4\sim5}S_2\PSi_2$ and $\PHi_5S_3\PSi_2$ are empty.
Since
\begin{gather*}
\PHi_{3\sim5}S_2\PSi_1=S_2[2\!:\!3,1],\
\PHi_5S_2\PSi_3=S_2[3,2],
\\
\PHi_{4\sim5}S_3\PSi_1=S_3[3,1],\
\PHi_5S_4\PSi_1=S_4[3,1],
\end{gather*}
then \eqref{eqn: example n=9 m=3 solve L} can be expressed as
\begin{equation*}
S_2[2\!:\!3,1]=0,\ S_2[3,2]=0,\ S_3[3,1]=0,\ S_4[3,1]=0,
\end{equation*}
where
\begin{align*}
S_2[2\!:\!3,1]
&=(
R_2[2\!:\!3,:]-R_1[2\!:\!3,:]CL
)R_1^{-1}[:,1]
=\begin{bmatrix}-1-l_{41}-l_{71}\\ -l_{71}\end{bmatrix},\\
S_2[3,2]
&=(
R_2[3,:]-R_1[3,:]CL
)R_1^{-1}[:,2]
=-1-l_{72},\\
S_3[3,1]
&=(
R_3[3,:]-R_2[3,:]CL-R_1[3,:]CAL
)R_1^{-1}[:,1]
=l_{41}-l_{81},\\
S_4[3,1]
&=(
R_4[3,:]-R_3[3,:]CL-R_2[3,:]CAL-R_1[3,:]CA^2L
)R_1^{-1}[:,1]\\
&=l_{51}-l_{91}.
\end{align*}
\textbf{Step 7: Final Solution for $L$}
\\
The constraints for $L$ are $l_{41}=-1$, $l_{72}=-1$, $l_{81}=-1$, $l_{71}=0$, and $l_{51}-l_{91}=0$.
The general form of $L$ is
\[
L=\begin{bmatrix}
*&*&*&-1&a&*&0&-1&-a\\
*&*&*&*&*&*&-1&*&*\\
*&*&*&*&*&*&*&*&*
\end{bmatrix}^\T,
\]
where $a \in \R$ and ``$*$'' denotes arbitrary values.

\section{Proofs of Main Theorems}

\subsection{Preparation}

\subsubsection{Criterion for the existence of a regular relative degree}

Here, we review the foundational results established in \cite{Cao2025rrdref[J]}.
Consider the partition of the reduced row-echelon form of the output controllability matrix $\Qoc^{[n]}(A,B,C)$ as follows,
\begin{equation}\label{eqn:Qoc rref}
T_\mathrm{rref}\Qoc^{[n]}
=\rref(\Qoc^{[n]})
=\ \ \begin{NiceArray}{ccccWc{17pt}}
G_1&*&\cdots&*&g_1\\
&G_2&\cdots&*&g_2\\
&&\ddots&\vdots&\vdots\\
&&&G_{n}&g_n\\
0&0&\cdots&0&\\
m&m&\cdots&m&
\CodeAfter\SubMatrix[{1-1}{5-4}][right-xshift=5pt]
\tikz \draw (2-|1) -- (2-|5);
\tikz \draw (3-|1) -- (3-|5);
\tikz \draw (4-|1) -- (4-|5);
\tikz \draw (5-|1) -- (5-|5);
\tikz \draw (1-|2) -- (6-|2);
\tikz \draw (1-|3) -- (6-|3);
\tikz \draw (1-|4) -- (6-|4);
\end{NiceArray},
\end{equation}
where $T_\mathrm{rref} \in \R_m^{m \times m}$ is the corresponding transformation matrix
and $G_i \in \R_{g_i}^{g_i \times m}$ for $i=1,2,\dots,n$.
Note that $g_i=0$ is possible for certain indices.
Let $G = \begin{bmatrix} G_1^\T & G_2^\T & \cdots & G_n^\T \end{bmatrix}^\T$.
The primary findings of \cite{Cao2025rrdref[J]} are summarized as follows.

\begin{lemma}\label{thm:rrdrefresult}
\emph{\cite{Cao2025rrdref[J]}}
For the system \eqref{eqn: xdot=Ax+Bu, y=Cx}, the following properties hold:
\begin{enumerate}

\item\label{thm: rrd iff detG not 0}
There exists an output transformation matrix $T_y \in \R^{m \times m}_m$
such that the transformed system $(A,B,T_yC)$
has a regular relative degree if and only if
\begin{equation}\label{eqn: sum gi=m and G invertible}
\Sigma_{i=1}^n g_i = m, \quad \det(G) \neq 0.
\end{equation}

\item\label{thm: solve Ty}
Suppose that \eqref{eqn: sum gi=m and G invertible} is satisfied.
Then, for any output transformation matrix $T_y \in \R_m^{m \times m}$,
the new system $(A,B,T_yC)$ has a regular relative degree if and only
if $T_y$ takes the form
\begin{equation}\label{eqn:Ty=PUTrref}
T_y = PUT_\mathrm{rref},
\end{equation}
where $P \in \R^{m \times m}_m$ is an arbitrary permutation matrix and
\begin{equation}\label{eqn: U=?}
U = \begin{bmatrix}
	U_1 & * & \cdots & * \\
	& U_2 & \cdots & * \\
	& & \ddots & \vdots \\
	& & & U_n
\end{bmatrix},
\text{ with $U_i \in \R^{g_i \times g_i}_{g_i}, \ i=1,2,\dots,n$.}
\end{equation}

\item\label{thm: relative degree=?}
Suppose that \eqref{eqn: sum gi=m and G invertible} is satisfied
and that $T_y \in \R^{m \times m}_m$ satisfies \eqref{eqn:Ty=PUTrref}.
Then, the relative degree of the system $(A,B,T_yC)$ is unique
up to the rearrangement of the set
$\{\underbrace{1,\dots,1}_{g_1\,\mathrm{times}},
\underbrace{2,\dots,2}_{g_2\,\mathrm{times}},\dots,
\underbrace{n,\dots,n}_{g_n\,\mathrm{times}}\}$,
where the element $i$ is absent if $g_i=0$.

\end{enumerate}
\end{lemma}

Based on the results in \cite{Cao2025rrdref[J]},
we develop several lemmas to facilitate our proof.
From Item~\ref{thm: solve Ty} of \Cref{thm:rrdrefresult},
the following result can be established.

\begin{lemma}\label{thm: solve Ty from an okay T}
Consider the system \eqref{eqn: xdot=Ax+Bu, y=Cx}.
If there exists a matrix $\tilde T \in \R_m^{m \times m}$ such that
\begin{equation*}
\tilde T \Qoc^{[n]} = \begin{bmatrix}
\tilde G_1 & * & \cdots & * \\
& \tilde G_2 & \cdots & * \\
& & \ddots & \vdots \\
& & & \tilde G_{n}
\end{bmatrix},
\end{equation*}
where each $\tilde G_i$ has $m$ columns and
$\begin{bmatrix} \tilde G_1^\T & \tilde G_2^\T & \cdots & \tilde G_n^\T \end{bmatrix}^\T$
is nonsingular,
then an output transformation matrix $T_y \in \R^{m \times m}_m$ exists
ensuring that the system $(A,B,T_yC)$ has a regular relative degree.
Moreover, the general form of $T_y$ is $T_y = PU\tilde T$,
where $P \in \R^{m \times m}_m$ is an arbitrary permutation matrix and
$U$ takes the form defined in \eqref{eqn: U=?}.
\end{lemma}

\begin{proof}
Since $\begin{bmatrix} \tilde G_1^\T & \tilde G_2^\T & \cdots & \tilde G_n^\T \end{bmatrix}^\T$ is nonsingular,
each $\tilde G_i$ possesses full row rank.
Thus, there exists a matrix $\tilde T_\mathrm{rref} \in \R^{m \times m}$ of the form
\begin{equation}\label{eqn: form of Ttilderref}
\tilde T_\mathrm{rref} = \begin{bmatrix}
\tilde U_1 & * & \cdots & * \\
& \tilde U_2 & \cdots & * \\
& & \ddots & \vdots \\
& & & \tilde U_{n}
\end{bmatrix}
\end{equation}
that converts $\tilde T \Qoc^{[n]}$ into its reduced row-echelon form defined in \eqref{eqn:Qoc rref}, specifically,
\begin{equation*}
\tilde T_\mathrm{rref} \tilde T \Qoc^{[n]} = \begin{bmatrix}
G_1 & * & \cdots & * \\
& G_2 & \cdots & * \\
& & \ddots & \vdots \\
& & & G_n
\end{bmatrix}.
\end{equation*}
By the uniqueness of the reduced row-echelon form,
we have $\tilde T_\mathrm{rref} \tilde T \Qoc^{[n]} = T_\mathrm{rref} \Qoc^{[n]}$.
Given that $\Qoc^{[n]}$ has full row rank,
it follows that $\tilde T_\mathrm{rref} \tilde T = T_\mathrm{rref}$.
By Item~\ref{thm: solve Ty} of \Cref{thm:rrdrefresult}, the general form of $T_y$ is
$T_y = PUT_\mathrm{rref} = PU\tilde T_\mathrm{rref}\tilde T$.
Considering the structure of $\tilde T_\mathrm{rref}$ in \eqref{eqn: form of Ttilderref},
the product $U\tilde T_\mathrm{rref}$ can be consolidated into a single matrix $U$
maintaining the form \eqref{eqn: U=?}.
This completes the proof.
\end{proof}

Utilizing Item~\ref{thm: rrd iff detG not 0} of \Cref{thm:rrdrefresult},
we establish the following lemma.

\begin{lemma}\label{thm: image cap=0}
Consider $i \in \{2, 3, \dots, n\}$ and a system \eqref{eqn: xdot=Ax+Bu, y=Cx}
satisfying \eqref{eqn: sum gi=m and G invertible}.
If the output controllability matrix satisfies
\[
\Qoc^{[i]} = \begin{bmatrix}
\hat G_1 & * & \cdots & * \\
& \hat G_2 & \cdots & * \\
& & \ddots & \vdots \\
& & & \hat G_i \\
\multicolumn{4}{c}{0}
\end{bmatrix},
\]
where each $\hat G_j$ has $m$ columns, then it follows that
\[
\im(\hat G_i^\T) \cap \left[ \im(\hat G_1^\T) + \im(\hat G_2^\T) + \dots + \im(\hat G_{i-1}^\T) \right] = \{0_m\}.
\]
\end{lemma}

\begin{proof}
Note that $\Qoc^{[n]} = \begin{bmatrix} \Qoc^{[i]} & * \end{bmatrix}$.
Referring to \eqref{eqn:Qoc rref} and the properties of the reduced row-echelon form,
we have $\im(\hat G_j^\T) = \im(G_j^\T)$ for $j=1,2,\dots,i$.
Item~\ref{thm: rrd iff detG not 0} of \Cref{thm:rrdrefresult} implies $\det(G) \neq 0$,
which further indicates that the row spaces of $G_j$ are linearly independent.
Thus, the intersection of these subspaces is trivial.
This completes the proof.
\end{proof}

\subsubsection{Relationship between $\Qi^{[i]}(A,B,C)$ and $\Qi^{[i]}(A+LC,B,C)$}

In this subsection, we establish the relationship between the invertibility matrices
$\Qi^{[i]}(A,B,C)$ and $\Qi^{[i]}(A+LC,B,C)$.
We begin by introducing a fundamental identity.

\begin{lemma}\label{thm: CA^iB=?}
For $i=1,2,\dots,n$, the following identity holds:
\[
CA^iB = C(A+LC)^iB + \Sigma_{j=0}^{i-1}(-CA^{i-1-j}L)C(A+LC)^jB.
\]
\end{lemma}

\begin{proof}
First, we prove by induction that
\begin{equation}\label{eqn: (A+LC)^i=?}
(A+LC)^i = A^i + \Sigma_{j=0}^{i-1}A^{i-1-j}LC(A+LC)^j,
\quad i=1,2,\dots,n.
\end{equation}
For $i=1$, \eqref{eqn: (A+LC)^i=?} obviously holds.
Assume that \eqref{eqn: (A+LC)^i=?} holds for some $i \ge 1$.
Then, for the case $i+1$, we have
\begin{align*}
(A+LC)^{i+1} &= (A+LC)(A+LC)^i\\
&= A(A+LC)^i + LC(A+LC)^i\\
&= A\left[A^i + \Sigma_{j=0}^{i-1}A^{i-1-j}LC(A+LC)^j\right] + LC(A+LC)^i\\
&= A^{i+1} + \Sigma_{j=0}^{i-1}A^{i-j}LC(A+LC)^j + LC(A+LC)^i\\
&= A^{i+1} + \Sigma_{j=0}^i A^{i-j}LC(A+LC)^j.
\end{align*}
By the principle of mathematical induction, \eqref{eqn: (A+LC)^i=?} holds for all $i=1,\dots,n$.
Substituting \eqref{eqn: (A+LC)^i=?} yields
\[
C(A+LC)^iB = CA^iB + \Sigma_{j=0}^{i-1}CA^{i-1-j}LC(A+LC)^jB,
\quad i=1,2,\dots,n.
\]
Rearranging the terms, we obtain:
$CA^iB = C(A+LC)^iB + \Sigma_{j=0}^{i-1}(-CA^{i-1-j}L)C(A+LC)^jB$.
This completes the proof.
\end{proof}

Using Lemma~\ref{thm: CA^iB=?}, we can characterize the relationship
between $\Qi^{[i]}(A,B,C)$ and $\Qi^{[i]}(A+LC,B,C)$.

\begin{lemma}
\label{thm: Qi(A,B,C)=Pio(A,C,L)*Qi(A+LC,B,C)}
For $i=2,3,\dots,n$, the following holds
\[
\Qi^{[i]}(A,B,C) = \Pio^{[i]}(L)\Qi^{[i]}(A+LC,B,C).
\]
\end{lemma}

\begin{proof}
Let $\mathrm{LHS}_{jk} \in \R^{m\times m}$ denote the block entries of the matrix $\Qi^{[i]}(A,B,C)$.
Specifically, for $j,k=1,2,\dots,i$, we have
\begin{equation}\label{eqn: LHS=what}
\mathrm{LHS}_{jk}=
\begin{cases}
0_{m\times m}, & \text{if } j+k \le i,\\
CA^{j+k-i-1}B, & \text{if } j+k \ge i+1.
\end{cases}
\end{equation}
Similarly, let $\mathrm{RHS}_{jk} \in \R^{m\times m}$ denote the block entries of the product
$\Pio^{[i]}(L)\Qi^{[i]}(A+LC,B,C)$.
We now compute $\mathrm{RHS}_{jk}$.
For $j=1$, it is easily seen that
\begin{equation}\label{eqn: RHSjk=? 1}
\mathrm{RHS}_{1k} = \begin{cases}
0_{m\times m}, & \text{if } k \le i-1,\\
CB, & \text{if } k=i.
\end{cases}
\end{equation}
For $j \ge 2$, the entry is given by
\begin{equation}\label{eqn: RHSjk=what*what}
\mathrm{RHS}_{jk} =
\begin{bmatrix}
-CA^{j-2}L & \cdots & -CL & I_m & 0_{m\times m(i-j)}
\end{bmatrix}
\begin{bmatrix}
0_{m(i-k)\times m} \\ CB \\ \vdots \\ C(A+LC)^{k-1}B
\end{bmatrix}.
\end{equation}
Consequently, for $j \ge 2$ and $j+k \le i+1$, we have
\begin{equation}\label{eqn: RHSjk=? 2}
\mathrm{RHS}_{jk} =
\begin{cases}
0_{m\times m}, & \text{if } j \ge 2, \ j+k \le i,\\
CB, & \text{if } j \ge 2, \ j+k = i+1.
\end{cases}
\end{equation}
From \eqref{eqn: RHSjk=what*what}, it follows that for $j \ge 2$ and $j+k \ge i+2$:
\begin{align*}
\mathrm{RHS}_{jk}
&= \begin{bmatrix}
-CA^{j+k-i-2}L & \cdots & -CL & I_m
\end{bmatrix}
\begin{bmatrix}
CB \\ C(A+LC)B \\ \vdots \\ C(A+LC)^{j+k-i-1}B
\end{bmatrix}\\
&= \Sigma_{l=0}^{j+k-i-2}(-CA^{j+k-i-2-l}L)C(A+LC)^lB
+ C(A+LC)^{j+k-i-1}B.
\end{align*}
Applying Lemma~\ref{thm: CA^iB=?}, we obtain
\begin{equation}\label{eqn: RHSjk=? 3}
\mathrm{RHS}_{jk} = CA^{j+k-i-1}B, \quad j \ge 2, \ j+k \ge i+2.
\end{equation}
Combining \eqref{eqn: RHSjk=? 1}, \eqref{eqn: RHSjk=? 2},
and \eqref{eqn: RHSjk=? 3} yields, for all $j,k=1,2,\dots,i$:
\begin{equation}\label{eqn: RHS=what}
\mathrm{RHS}_{jk}
= \begin{cases}
0_{m\times m}, & \text{if } j+k \le i,\\
CA^{j+k-i-1}B, & \text{if } j+k \ge i+1.
\end{cases}
\end{equation}
Comparing \eqref{eqn: LHS=what} and \eqref{eqn: RHS=what},
we conclude $\mathrm{LHS}_{jk} = \mathrm{RHS}_{jk}$.
This completes the proof.
\end{proof}

The significance of
Lemma~\ref{thm: Qi(A,B,C)=Pio(A,C,L)*Qi(A+LC,B,C)} lies in
the fact that the entries of $\Pio^{[i]}(L)$ are linear with
respect to the Luenberger gain $L$.

\begin{remark}
In Section 4.6.2 of \cite{Zhoubook}, a similar relationship
is presented between $\Qi^{[i]}(A+LC,B,C)$ and $\Qi^{[i]}(A,B,C)$:
\[
\Qi^{[i]}(A+LC,B,C) = \mathit\Pi_\mathrm o^{[i]}(L)\Qi^{[i]}(A,B,C),
\]
where
\[
\mathit\Pi_\mathrm o^{[i]}(L) = \begin{bmatrix}
I_m & & & & \\
CL & I_m & & & \\
C(A+LC)L & CL & I_m & & \\
\vdots & \ddots & \ddots & \ddots & \\
C(A+LC)^{i-2}L & \cdots & C(A+LC)L & CL & I_m
\end{bmatrix}.
\]
However, the entries of $\mathit\Pi_\mathrm o^{[i]}(L)$
are not linear in $L$.
The relationship between $\Pio^{[i]}(L)$
and $\mathit\Pi_\mathrm o^{[i]}(L)$
is given by
$\Pio^{[i]}(L)\mathit\Pi_\mathrm o^{[i]}(L) = I_{mi}$.
\end{remark}

\subsubsection{Notational conventions}

In this subsection, we establish the notational conventions that
will be employed throughout the subsequent proofs.
Given the sequence $\psi_i$, $i=1,2,\dots,\nsharp$,
a matrix $X \in \R^{m \times m}$ is partitioned as
\[
X=\begin{bmatrix}
X_{11}&X_{12}&\cdots&X_{1\nsharp}\\
X_{21}&X_{22}&\cdots&X_{2\nsharp}\\
\vdots&\vdots&\ddots&\vdots\\
X_{\nsharp1}&X_{\nsharp2}&\cdots&X_{\nsharp\nsharp}
\end{bmatrix},
\]
where $X_{ij} \in \R^{\psi_i \times \psi_j}$.
It should be noted that some submatrices $X_{ij}$ may be empty.
We define the notation $X|_{p_1\sim p_2}^{q_1\sim q_2}$,
where $p_1, p_2, q_1, q_2 \in \mathbf Z$, as follows
\begin{itemize}
\item If $1 \le p_1 \le p_2 \le \nsharp$ and $1 \le q_1 \le q_2 \le \nsharp$,
then we denote
\[
X|_{p_1\sim p_2}^{q_1\sim q_2}
=\begin{bmatrix}
X_{p_1q_1}&X_{p_1,q_1+1}&\cdots&X_{p_1q_2}\\
X_{p_1+1,q_1}&X_{p_1+1,q_1+1}&\cdots&X_{p_1+1,q_2}\\
\vdots&\vdots&\ddots&\vdots\\
X_{p_2q_1}&X_{p_2,q_1+1}&\cdots&X_{p_2q_2}
\end{bmatrix};
\]
\item If the conditions $1 \le p_1 \le p_2 \le \nsharp$ and $1 \le q_1 \le q_2 \le \nsharp$ are not satisfied,
then $X|_{p_1\sim p_2}^{q_1\sim q_2}$ is regarded as an empty matrix.
\end{itemize}
Furthermore, we define $X|_p^{q_1\sim q_2} = X|_{p\sim p}^{q_1\sim q_2}$,
$X|_{p_1\sim p_2}^q = X|_{p_1\sim p_2}^{q\sim q}$,
and $X|_p^q = X|_{p\sim p}^{q\sim q}$.
In particular, the following specific notations are used
\begin{itemize}
\item $0|_{p_1\sim p_2}^{q_1\sim q_2}$:
the zero matrix of the same dimensions as
$X|_{p_1\sim p_2}^{q_1\sim q_2}$;
\item $0|_{p_1\sim p_2}^q$:
shorthand for $0|_{p_1\sim p_2}^{q\sim q}$;
\item $0|_{p_1\sim p_2}$:
shorthand for $0|_{p_1\sim p_2}^{1\sim\nsharp}$;
\item $I|_p$:
the identity matrix of the same dimensions as $X|_p^p$;
\item $*|_{p_1\sim p_2}^{q_1\sim q_2}$:
a matrix whose specific entries are of no concern,
possessing the same dimensions as $X|_{p_1\sim p_2}^{q_1\sim q_2}$;
\item $*|_{p_1\sim p_2}$:
shorthand for $*|_{p_1\sim p_2}^{1\sim\nsharp}$.
\end{itemize}
Note that for any matrix $X \in \R^{m \times m}$,
it holds that $\PHi_{p_1\sim p_2} X \PSi_{q} = X|_{p_1\sim p_2}^{q}$.
Consequently, the main result \eqref{eqn: main result} can be rewritten as
\begin{equation}\label{eqn: rewrite main result}
S_i|_{i+j\sim \nsharp}^j=0,\quad
j=1,2,\dots,\nsharp-i,\
i=2,3,\dots,\nsharp-1.
\end{equation}
The notation $X|_{p_1\sim p_2}^{q}$ will be utilized frequently in the following proofs,
and the form \eqref{eqn: rewrite main result} will be used for the proof of Theorem~\ref{thm: theorem}.

\subsubsection{Definition of $H_i$}

We now introduce the definition of $H_i$ for $i=1,2,\dots,\nsharp$, which
will facilitate the subsequent proofs.
Define
\begin{equation}\label{eqn: define H_i}
H_i = F_i - \Sigma_{j=2}^i S_j|_i^{i+1-j} H_{i+1-j}
\in \R^{\psi_i\times m},
\quad i=1,2,\dots,\nsharp,
\end{equation}
which is equivalent to
\begin{equation}\label{eqn: define H_i version 2}
H_i = F_i - \Sigma_{j=1}^{i-1} S_{i+1-j}|_i^j H_j,
\quad i=1,2,\dots,\nsharp.
\end{equation}
It should be noted that for $i=1$, the summation terms
in \eqref{eqn: define H_i} and \eqref{eqn: define H_i version 2} do not exist.
Thus, $H_1$ is defined as $H_1 = F_1$,
while the remaining $H_i$ are defined recursively.

Next, we establish that the matrix
$\begin{bmatrix}H_1^\T & H_2^\T & \cdots & H_{\nsharp}^\T\end{bmatrix}^\T$
is nonsingular.
Given that $H_1 = F_1$, we proceed by induction. Assume that
for some $i \in \{1, 2, \dots, \nsharp-1\}$, there holds
\begin{equation}\label{eqn: H_i inductive hypothesis}
\begin{bmatrix}H_1\\ H_2\\ \vdots\\ H_i\end{bmatrix}
=\begin{bmatrix}
I|_1&&&\\
*&I|_2&\\
\vdots&\vdots&\ddots&\\
*&*&\cdots&I|_i\end{bmatrix}
\begin{bmatrix}F_1\\ F_2\\ \vdots\\ F_i\end{bmatrix}.
\end{equation}
Considering the recursive definition \eqref{eqn: define H_i version 2} for $i+1$,
we have
\begin{align}
H_{i+1}
\notag&
=F_{i+1} - \Sigma_{j=1}^i S_{i+2-j}|_{i+1}^j H_j
=F_{i+1} + \begin{bmatrix}* & * & \cdots & *\end{bmatrix}
\begin{bmatrix}H_1\\ H_2\\ \vdots\\ H_i\end{bmatrix}
\\
\notag&
=F_{i+1} + \begin{bmatrix}* & * & \cdots & *\end{bmatrix}
\begin{bmatrix}
*&&&\\
*&*&\\
\vdots&\vdots&\ddots&\\
*&*&\cdots&*\end{bmatrix}
\begin{bmatrix}F_1\\ F_2\\ \vdots\\ F_i\end{bmatrix}
\\
&\label{eqn: Hi+1=?}
=F_{i+1} + \begin{bmatrix}* & * & \cdots & *\end{bmatrix}
\begin{bmatrix}F_1\\ F_2\\ \vdots\\ F_i\end{bmatrix}
=\begin{bmatrix}* & * & \cdots & * & I|_{i+1}\end{bmatrix}
\begin{bmatrix}F_1\\ F_2\\ \vdots\\ F_i\\ F_{i+1}\end{bmatrix}.
\end{align}
Combining \eqref{eqn: Hi+1=?} with the inductive hypothesis
\eqref{eqn: H_i inductive hypothesis} yields
\[
\begin{bmatrix}H_1\\ H_2\\ \vdots\\ H_i\\ H_{i+1}\end{bmatrix}
=\begin{bmatrix}
I|_1&&&&\\
*&I|_2&&\\
\vdots&\vdots&\ddots&&\\
*&*&\cdots&I|_i&\\
*&*&\cdots&*&I|_{i+1}\end{bmatrix}
\begin{bmatrix}F_1\\ F_2\\ \vdots\\ F_i\\ F_{i+1}\end{bmatrix}.
\]
By the principle of mathematical induction, we obtain
\[
\begin{bmatrix}H_1\\ H_2\\ \vdots\\ H_{\nsharp}\end{bmatrix}
=\begin{bmatrix}
I|_1&&&\\
*&I|_2&\\
\vdots&\vdots&\ddots&\\
*&*&\cdots&I|_{\nsharp}
\end{bmatrix}
\begin{bmatrix}F_1\\ F_2\\ \vdots\\ F_{\nsharp}\end{bmatrix}.
\]
Consequently, the nonsingularity of
$\begin{bmatrix}F_1^\T & F_2^\T & \cdots & F_{\nsharp}^\T\end{bmatrix}^\T$
guarantees the nonsingularity of
$\begin{bmatrix}H_1^\T & H_2^\T & \cdots & H_{\nsharp}^\T\end{bmatrix}^\T$
as well.

\subsubsection{Expansion of $S \Qi^{[\nsharp]}(A+LC, B, R_1C)$}

For brevity, we denote $\Ao = A + LC$.
By applying Lemma~\ref{thm: Qi(A,B,C)=Pio(A,C,L)*Qi(A+LC,B,C)}
and the definition of $S$ in \eqref{eqn: define S}, we obtain
\begin{equation}\label{eqn: R*Qisharp=S*Qi(Ao,B,R1C)}
\begin{aligned}
R\Qi^{[\nsharp]}(A,B,C)
&
= R\Pio^{[\nsharp]} \Qi^{[\nsharp]}(\Ao,B,C)
\\
&
= \left[ R\Pio^{[\nsharp]} (I_{\nsharp} \otimes R_1^{-1}) \right]
\left[ (I_{\nsharp} \otimes R_1) \Qi^{[\nsharp]}(\Ao,B,C) \right]
\\
&
= S \Qi^{[\nsharp]}(\Ao, B, R_1C).
\end{aligned}
\end{equation}
Combining \eqref{eqn: R*Qisharp=S*Qi(Ao,B,R1C)} with \eqref{eqn: RQi=G1~n} yields
\[
S \Qi^{[\nsharp]}(\Ao, B, R_1C)
= \begin{bmatrix}
F_1 & * & \cdots & * \\
& F_2 & \cdots & * \\
& & \ddots & \vdots \\
& & & F_{\nsharp}
\end{bmatrix},
\]
which can be expanded as
\begin{align*}
&\begin{bmatrix} S_{\nsharp} & \cdots & S_2 & I_m \end{bmatrix}
\begin{bmatrix}
&& R_1CB \\
& \adots & \vdots \\
R_1CB & \cdots & R_1C\Ao^{\nsharp-1}B
\end{bmatrix}
= \begin{bmatrix}
F_1 & \cdots & * \\
& \ddots & \vdots \\
&& F_{\nsharp}
\end{bmatrix}.
\end{align*}
Consequently, we have
\[
\begin{bmatrix} S_{\nsharp} & \cdots & S_2 & I_m \end{bmatrix}
\begin{bmatrix}
0_{m(\nsharp-i) \times m} \\ R_1CB \\ R_1C\Ao B \\ \vdots \\ R_1C\Ao^{i-1}B
\end{bmatrix}
= \begin{bmatrix}
*|_{1\sim i-1} \\ F_i \\ 0|_{i+1\sim\nsharp}
\end{bmatrix},
\quad
i=1,2,\dots,\nsharp.
\]
This leads to the following relation
\[
R_1 C \Ao^{i-1}B +
\Sigma_{j=2}^i S_j R_1 C \Ao^{i-j}B
= \begin{bmatrix}
*|_{1\sim i-1} \\ F_i \\ 0|_{i+1\sim\nsharp}
\end{bmatrix},
\quad i=1,2,\dots,\nsharp.
\]
By extracting the block rows from index $i$ to $\nsharp$, we arrive at
\begin{equation}\label{eqn: important}
R_1|_{i\sim\nsharp} C \Ao^{i-1}B + \Sigma_{j=2}^i S_j|_{i\sim\nsharp} R_1 C \Ao^{i-j}B
= \begin{bmatrix} F_i \\ 0|_{i+1\sim\nsharp} \end{bmatrix},
\quad i=1,2,\dots,\nsharp.
\end{equation}
The expression \eqref{eqn: important} serves as the fundamental starting point
for the derivations in the subsequent two subsections.
Note that for $i=1$, the summation term in \eqref{eqn: important} is null.

\subsubsection{Preparation for the induction}

In the subsequent two subsections,
we employ mathematical induction to prove
both the sufficiency and necessity parts of Theorem~\ref{thm: theorem}.
To this end,
we establish two preliminary results that are essential to both parts.
First, we specify the initial condition.
Recalling the definition $H_1=F_1$,
and setting $i=1$ in \eqref{eqn: important}, we obtain
\begin{equation}\label{eqn: case i=1}
R_1|_{1\sim\nsharp}CB=\begin{bmatrix}H_1\\ 0|_{2\sim\nsharp}\end{bmatrix}.
\end{equation}
Second, we derive some equations for the case $i+1$,
which will be referenced frequently in the forthcoming proofs.
In \eqref{eqn: important}, the case $i+1$ corresponds to
\[
R_1|_{i+1\sim\nsharp}C\Ao^iB+\Sigma_{j=2}^{i+1}S_j|_{i+1\sim\nsharp}R_1C\Ao^{i+1-j}B
=\begin{bmatrix}F_{i+1}\\ 0|_{i+2\sim\nsharp}\end{bmatrix}.
\]
For $j=2,3,\dots,i+1$, we partition $S_j|_{i+1\sim\nsharp}$ and $R_1C\Ao^{i+1-j}B$ as follows
\begin{gather}\label{eqn: S=S big}
S_j|_{i+1\sim\nsharp}
=\begin{bmatrix}
\multirow{2}{*}{$S_j|_{i+1\sim\nsharp}^{1\sim i+1-j}$}
&S_j|^{i+2-j}_{i+1}
&\multirow{2}{*}{$S_j|_{i+1\sim\nsharp}^{i+3-j\sim\nsharp}$}\\
&S_j|_{i+2\sim\nsharp}^{i+2-j}&
\end{bmatrix},
\\ \label{eqn:partition R1CAowhatB}
R_1C\Ao^{i+1-j}B=
\begin{bmatrix}
R_1|_{1\sim j-1}C\Ao^{i+1-j}B\\
R_1|_{j\sim\nsharp}C\Ao^{i+1-j}B
\end{bmatrix}.
\end{gather}
It then follows that
\begin{gather}
\notag
R_1|_{i+1\sim\nsharp}C\Ao^iB
+\Sigma_{j=2}^{i+1}
\begin{bmatrix}
\multirow{2}{*}{$S_j|_{i+1\sim\nsharp}^{1\sim i+1-j}$}
&S_j|^{i+2-j}_{i+1}
&\multirow{2}{*}{$S_j|_{i+1\sim\nsharp}^{i+3-j\sim\nsharp}$}\\
&S_j|_{i+2\sim\nsharp}^{i+2-j}&
\end{bmatrix}
\\ \label{eqn: TCAB+[S big]TCAB=F 0}
{}\times
\begin{bmatrix}
R_1|_{1\sim j-1}C\Ao^{i+1-j}B\\
R_1|_{j\sim\nsharp}C\Ao^{i+1-j}B
\end{bmatrix}
=\begin{bmatrix}F_{i+1}\\ 0|_{i+2\sim\nsharp}\end{bmatrix}.
\end{gather}
Note that in
\eqref{eqn: S=S big} and \eqref{eqn: TCAB+[S big]TCAB=F 0},
the submatrix $S_j|_{i+1\sim\nsharp}^{1\sim i+1-j}$ is empty when $j=i+1$.

\subsection{Proof for Theorem~\ref{thm: theorem} (Sufficiency),
Theorem~\ref{thm: theorem solve Ty}, and
Proposition~\ref{thm: relative degree of (A+LC,B,TyC)}}

To prove the sufficiency part of Theorem~\ref{thm: theorem},
we assume that \eqref{eqn: rewrite main result} holds.
Based on the initial condition \eqref{eqn: case i=1} for $i=1$,
we establish the inductive hypothesis that,
for some $i \in \{1, 2, \dots, \nsharp-1\}$, there holds
\begin{equation}\label{eqn: sufficiency inductive hypothesis}
R_1|_{j\sim\nsharp} C \Ao^{j-1} B
= \begin{bmatrix} H_j \\ 0|_{j+1\sim\nsharp} \end{bmatrix},
\quad j = 1, 2, \dots, i.
\end{equation}

First, starting from the prerequisite \eqref{eqn: rewrite main result},
we derive the conditions \eqref{eqn:S=0tocondition1} and \eqref{eqn:S=0tocondition2} below,
which imply that the terms $S_j|_{i+1\sim\nsharp}^{1\sim i+1-j}$
and $S_j|_{i+2\sim\nsharp}^{i+2-j}$ in \eqref{eqn: S=S big} vanish.
The condition \eqref{eqn: rewrite main result} can be expressed as
\begin{equation}\label{eqn: sufficiency use prerequist}
S_j|_{j+k\sim\nsharp}^k = 0, \quad k = 1, 2, \dots, \nsharp-j, \ j = 2, 3, \dots, \nsharp-1.
\end{equation}
On one hand, if $i \ge 2$, then for $j+k \le i+1$,
the matrix $S_j|_{j+k\sim\nsharp}^k$ in \eqref{eqn: sufficiency use prerequist}
can be partitioned as
\begin{equation}\label{eqn: sufficiency partition S}
S_j|_{j+k\sim\nsharp}^k =
\begin{bmatrix}
S_j|_{j+k\sim i}^k \\
S_j|_{i+1\sim\nsharp}^k
\end{bmatrix},
\quad
k = 1, 2, \dots, i+1-j, \
j = 2, 3, \dots, i.
\end{equation}
Note that in \eqref{eqn: sufficiency partition S},
the submatrix $S_j|_{j+k\sim i}^k$ is empty when $j+k = i+1$.
Consequently, if $i \ge 2$, we have
\[
S_j|_{i+1\sim\nsharp}^k = 0,
\quad
k = 1, 2, \dots, i+1-j, \
j = 2, 3, \dots, i.
\]
Arranging $S_j|_{i+1\sim\nsharp}^k$ in a row along the index $k$ yields
\begin{equation}\label{eqn:S=0tocondition1}
S_j|_{i+1\sim\nsharp}^{1\sim i+1-j}
= \begin{bmatrix}
S_j|_{i+1\sim\nsharp}^1 & S_j|_{i+1\sim\nsharp}^2 & \cdots & S_j|_{i+1\sim\nsharp}^{i+1-j}
\end{bmatrix}
= 0,
\quad j = 2, 3, \dots, i.
\end{equation}
On the other hand, by setting $k = i+2-j$ in \eqref{eqn: sufficiency use prerequist},
since $1 \le k \le \nsharp-j$ implies $1 \le i+2-j \le \nsharp-j$
(and thus $i \le \nsharp-2$ and $j \le i+1$), we obtain
\begin{equation}\label{eqn:S=0tocondition2}
S_j|_{i+2\sim\nsharp}^{i+2-j} = 0,
\quad j = 2, 3, \dots, i+1.
\end{equation}

Next, we employ \eqref{eqn:S=0tocondition1} and \eqref{eqn:S=0tocondition2}
to advance the derivation.
By substituting these conditions into \eqref{eqn: S=S big},
the matrix $S_j|_{i+1\sim\nsharp}$ becomes
\begin{equation}\label{eqn:Sj...1}
S_j|_{i+1\sim\nsharp}
= \begin{bmatrix}
\multirow{2}{*}{$0|_{i+1\sim\nsharp}^{1\sim i+1-j}$}
& S_j|^{i+2-j}_{i+1}
& \multirow{2}{*}{$S_j|_{i+1\sim\nsharp}^{i+3-j\sim\nsharp}$} \\
& 0|_{i+2\sim\nsharp}^{i+2-j} &
\end{bmatrix},
\quad j = 2, 3, \dots, i+1.
\end{equation}
The inductive hypothesis \eqref{eqn: sufficiency inductive hypothesis} can be rewritten as
\begin{equation}\label{eqn:R1...1}
R_1|_{j\sim\nsharp} C \Ao^{i+1-j} B =
\begin{bmatrix}
H_{i+2-j} \\ 0|_{i+3-j\sim\nsharp}
\end{bmatrix},
\quad j = 2, 3, \dots, i+1.
\end{equation}
By incorporating \eqref{eqn:Sj...1} and \eqref{eqn:R1...1} into \eqref{eqn: TCAB+[S big]TCAB=F 0}
and treating terms $R_1|_{1\sim j-1} C \Ao^{i+1-j} B$ and $S_j|_{i+1\sim\nsharp}^{i+3-j\sim\nsharp}$ as ``$*$'',
the equation \eqref{eqn: TCAB+[S big]TCAB=F 0} for the case $i+1$ becomes
\begin{gather*}
R_1|_{i+1\sim\nsharp} C \Ao^i B + \Sigma_{j=2}^{i+1}
\begin{bmatrix}
\multirow{2}{*}{$0|_{i+1\sim\nsharp}^{1\sim i+1-j}$}
& S_j|^{i+2-j}_{i+1}
& \multirow{2}{*}{$*|_{i+1\sim\nsharp}^{i+3-j\sim\nsharp}$} \\
& 0|_{i+2\sim\nsharp}^{i+2-j} &
\end{bmatrix} \\
{}\times
\begin{bmatrix}
*|_{1\sim i+1-j} \\ H_{i+2-j} \\ 0|_{i+3-j\sim\nsharp}
\end{bmatrix}
= \begin{bmatrix} F_{i+1} \\ 0|_{i+2\sim\nsharp} \end{bmatrix}.
\end{gather*}
This simplifies to
\[
R_1|_{i+1\sim\nsharp} C \Ao^i B + \Sigma_{j=2}^{i+1}
\begin{bmatrix}
S_j|^{i+2-j}_{i+1} H_{i+2-j} \\ 0|_{i+2\sim\nsharp}
\end{bmatrix}
= \begin{bmatrix} F_{i+1} \\ 0|_{i+2\sim\nsharp} \end{bmatrix},
\]
leading to
\begin{equation}\label{eqn: Ti+1~nCAB=?}
R_1|_{i+1\sim\nsharp} C \Ao^i B
= \begin{bmatrix}
F_{i+1} - \Sigma_{j=2}^{i+1} S_j|^{i+2-j}_{i+1} H_{i+2-j} \\ 0|_{i+2\sim\nsharp}
\end{bmatrix}.
\end{equation}
By recalling the definition of $H_i$ in \eqref{eqn: define H_i} and considering the case $i+1$, it yields
\begin{equation}\label{eqn: Hi+1 definition}
H_{i+1} = F_{i+1} - \Sigma_{j=2}^{i+1} S_j|^{i+2-j}_{i+1} H_{i+2-j}.
\end{equation}
Substituting \eqref{eqn: Hi+1 definition} into \eqref{eqn: Ti+1~nCAB=?} yields
\begin{equation}\label{eqn:R1CAiB=Hi+1}
R_1|_{i+1\sim\nsharp} C \Ao^i B
= \begin{bmatrix}
H_{i+1} \\ 0|_{i+2\sim\nsharp}
\end{bmatrix}.
\end{equation}
Thus, the inductive hypothesis \eqref{eqn: sufficiency inductive hypothesis}
holds for $j = i+1$.
By the principle of mathematical induction, \eqref{eqn: sufficiency inductive hypothesis}
holds for all $j = 1, 2, \dots, \nsharp$, i.e.,
\begin{equation}\label{eqn:R1=Hi0}
R_1|_{i\sim\nsharp} C \Ao^{i-1} B
= \begin{bmatrix} H_i \\ 0|_{i+1\sim\nsharp} \end{bmatrix},
\quad i = 1, 2, \dots, \nsharp,
\end{equation}
where the matrix $0|_{i+1\sim\nsharp}$ is empty when $i = \nsharp$.

The result \eqref{eqn:R1=Hi0} implies
\begin{align*}
& \Qoc^{[\nsharp]}(A+LC, B, R_1C)
= \begin{bmatrix} R_1CB & R_1C\Ao B & \cdots & R_1C\Ao^{\nsharp-1}B \end{bmatrix} \\
&\quad =
\def\entryIi{\Block[draw]{4-1}{R_1|_{1\sim\nsharp}CB}}
\def\entryIIi{\Block[draw]{3-1}{R_1|_{2\sim\nsharp}C\Ao B}}
\def\entryIViv{\Block[draw]{}{R_1|_{\nsharp}C\Ao^{\nsharp-1}B}}
\setlength{\extrarowheight}{1mm}
\left[
\begin{NiceArray}{cccc}[margin,rules/color=black,no-cell-nodes]
\entryIi & * & \cdots & * \\
         & \entryIIi & \cdots & \vdots     \\
         &           & \ddots & * \\
         &           & \cdots & \entryIViv \\
\end{NiceArray}
\right]
\\
&\quad
= \begin{bmatrix}
H_1 & * & \cdots & * \\
& H_2 & \cdots & * \\
& & \ddots & \vdots \\
& & & H_{\nsharp}
\end{bmatrix}.
\end{align*}
Given the nonsingularity of
$\begin{bmatrix} H_1^\T & H_2^\T & \cdots & H_{\nsharp}^\T \end{bmatrix}^\T$,
the decoupling matrix
$\Gamma(A+LC, B, R_1C) = \begin{bmatrix} H_1^\T & H_2^\T & \cdots & H_{\nsharp}^\T \end{bmatrix}^\T$
is nonsingular.
This establishes that the system $(A+LC, B, R_1C)$ has a regular relative degree,
completing the proof for the sufficiency of Theorem~\ref{thm: theorem}.

Furthermore, we have
\[
R_1 \Qoc^{[\nsharp]}(A+LC, B, C)
= \Qoc^{[\nsharp]}(A+LC, B, R_1C)
= \begin{bmatrix}
H_1 & * & \cdots & * \\
& H_2 & \cdots & * \\
& & \ddots & \vdots \\
& & & H_{\nsharp}
\end{bmatrix}.
\]
According to Lemma~\ref{thm: solve Ty from an okay T},
the general form of $T_y$ is indeed \eqref{eqn: Ty=PUR1}.
This completes the proof of Theorem~\ref{thm: theorem solve Ty}.

Finally, the relative degree of the system $(A+LC, B, R_1C)$ is
given by \eqref{eqn: relative degree psi1...psinsharp}.
From Item~\ref{thm: relative degree=?} of Lemma~\ref{thm:rrdrefresult},
it follows that for any $T_y$ ensuring a regular relative degree for
the system $(A+LC, B, T_yC)$, the relative degree is given by
\eqref{eqn: relative degree psi1...psinsharp}.
This completes the proof of Proposition~\ref{thm: relative degree of (A+LC,B,TyC)}.

\subsection{Proof for the Necessity of Theorem~\ref{thm: theorem}}

To prove the necessity of Theorem~\ref{thm: theorem},
we assume that there exists a matrix $T_y \in \R^{m \times m}_m$
such that the system $(A+LC, B, T_yC)$ possesses a regular relative degree.
According to Item~\ref{thm: rrd iff detG not 0} of Lemma~\ref{thm:rrdrefresult},
the system $(A+LC, B, C)$ must satisfy \eqref{eqn: sum gi=m and G invertible}.
Based on the initial condition \eqref{eqn: case i=1} for $i=1$,
we establish the inductive hypothesis that,
for some $i \in \{1, 2, \dots, \nsharp-2\}$, there holds
\begin{equation}\label{eqn: necessity inductive hypothesis}
\begin{cases}
R_1|_{j\sim\nsharp} C \Ao^{j-1} B
= \begin{bmatrix} H_j \\ 0|_{j+1\sim\nsharp} \end{bmatrix},
\quad j = 1, 2, \dots, i,
\\
S_j|_{j+k\sim\nsharp}^k = 0,
\quad k = 1, 2, \dots, i+1-j, \
j = 2, 3, \dots, i.
\end{cases}
\end{equation}
Note that for $i=1$, the matrix $S_j|_{j+k\sim\nsharp}^k$ in \eqref{eqn: necessity inductive hypothesis} is empty.
For $i \ge 2$, since $j+k \le i+1$,
we obtain the following matrix partition:
\begin{equation}\label{eqn: necessity partition S}
S_j|_{j+k\sim\nsharp}^k
= \begin{bmatrix}
S_j|_{j+k\sim i}^k \\ S_j|_{i+1\sim\nsharp}^k
\end{bmatrix},
\quad k = 1, 2, \dots, i+1-j, \ j = 2, 3, \dots, i.
\end{equation}
It should be noted that
the submatrix $S_j|_{j+k\sim i}^k$ in \eqref{eqn: necessity partition S} is empty when $j+k=i+1$.
From \eqref{eqn: necessity partition S}, it follows that
\[
S_j|_{i+1\sim\nsharp}^k = 0,
\quad k = 1, 2, \dots, i+1-j, \ j = 2, 3, \dots, i.
\]
Arranging $S_j|_{i+1\sim\nsharp}^k$ in a row along the index $k$ yields
\[
S_j|_{i+1\sim\nsharp}^{1\sim i+1-j}
= \begin{bmatrix}
S_j|_{i+1\sim\nsharp}^1 & S_j|_{i+1\sim\nsharp}^2 & \cdots & S_j|_{i+1\sim\nsharp}^{i+1-j}
\end{bmatrix}
= 0, \quad j = 2, 3, \dots, i.
\]
By virtue of \eqref{eqn: S=S big}, the matrix $S_j|_{i+1\sim\nsharp}$ can be expressed as
\begin{equation}\label{eqn:Sj...2}
S_j|_{i+1\sim\nsharp}
= \begin{bmatrix}
\multirow{2}{*}{$0|_{i+1\sim\nsharp}^{1\sim i+1-j}$}
& S_j|^{i+2-j}_{i+1}
& \multirow{2}{*}{$S_j|_{i+1\sim\nsharp}^{i+3-j\sim\nsharp}$} \\
& S_j|_{i+2\sim\nsharp}^{i+2-j} &
\end{bmatrix},
\quad j = 2, 3, \dots, i+1.
\end{equation}
Applying the inductive hypothesis \eqref{eqn: necessity inductive hypothesis}, we have
\begin{equation}\label{eqn:R1...2}
R_1|_{j\sim\nsharp} C \Ao^{i+1-j} B =
\begin{bmatrix}
H_{i+2-j} \\ 0|_{i+3-j\sim\nsharp}
\end{bmatrix},
\quad j = 2, 3, \dots, i+1.
\end{equation}
By substituting \eqref{eqn:Sj...2} and \eqref{eqn:R1...2} into \eqref{eqn: TCAB+[S big]TCAB=F 0}
and treating terms $R_1|_{1\sim j-1} C \Ao^{i+1-j} B$ and $S_j|_{i+1\sim\nsharp}^{i+3-j\sim\nsharp}$ as ``$*$'',
the equation \eqref{eqn: TCAB+[S big]TCAB=F 0} for the case $i+1$ becomes
\begin{gather*}
R_1|_{i+1\sim\nsharp} C \Ao^i B
+ \Sigma_{j=2}^{i+1}
\begin{bmatrix}
\multirow{2}{*}{$0|_{i+1\sim\nsharp}^{1\sim i+1-j}$}
& S_j|^{i+2-j}_{i+1}
& \multirow{2}{*}{$*|_{i+1\sim\nsharp}^{i+3-j\sim\nsharp}$} \\
& S_j|_{i+2\sim\nsharp}^{i+2-j} &
\end{bmatrix} \\
{}\times
\begin{bmatrix}
*|_{1\sim i+1-j} \\
H_{i+2-j} \\
0|_{i+3-j\sim\nsharp}
\end{bmatrix}
= \begin{bmatrix} F_{i+1} \\ 0|_{i+2\sim\nsharp} \end{bmatrix}.
\end{gather*}
This leads to
\[
R_1|_{i+1\sim\nsharp} C \Ao^i B + \Sigma_{j=2}^{i+1}
\begin{bmatrix}
S_j|^{i+2-j}_{i+1} H_{i+2-j} \\
S_j|_{i+2\sim\nsharp}^{i+2-j} H_{i+2-j}
\end{bmatrix}
= \begin{bmatrix} F_{i+1} \\ 0|_{i+2\sim\nsharp} \end{bmatrix},
\]
and therefore
\[
R_1|_{i+1\sim\nsharp} C \Ao^i B
= \begin{bmatrix}
F_{i+1} - \Sigma_{j=2}^{i+1} S_j|^{i+2-j}_{i+1} H_{i+2-j} \\
- \Sigma_{j=2}^{i+1} S_j|_{i+2\sim\nsharp}^{i+2-j} H_{i+2-j}
\end{bmatrix}.
\]
Recalling the recursive definition of $H_i$ in \eqref{eqn: define H_i} for $i+1$:
\[
H_{i+1} = F_{i+1} - \Sigma_{j=2}^{i+1} S_j|^{i+2-j}_{i+1} H_{i+2-j},
\]
it follows that
\[
R_1|_{i+1\sim\nsharp} C \Ao^i B
= \begin{bmatrix}
H_{i+1} \\
- \Sigma_{j=2}^{i+1} S_j|_{i+2\sim\nsharp}^{i+2-j} H_{i+2-j}
\end{bmatrix}.
\]
Incorporating the inductive hypothesis \eqref{eqn: necessity inductive hypothesis}, we obtain
\begin{align*}
&\Qoc^{[i+1]}(\Ao, B, R_1C)
= \begin{bmatrix} R_1CB & R_1C\Ao B & \cdots & R_1C\Ao^{i-1}B & R_1C\Ao^iB \end{bmatrix} \\
&=
\def\entryIi{\Block[draw]{5-1}{R_1|_{1\sim\nsharp}CB}}
\def\entryIIi{\Block[draw]{4-1}{R_1|_{2\sim\nsharp}C\Ao B}}
\def\entryIViv{\Block[draw]{2-1}{R_1|_{i\sim\nsharp}C\Ao^{i-1}B}}
\def\entryVv{\Block[draw]{}{R_1|_{i+1\sim\nsharp}C\Ao^iB}}
\setlength{\extrarowheight}{1mm}
\left[\begin{NiceArray}{ccccc}[margin,rules/color=black,no-cell-nodes]
\entryIi & * & \cdots & * & * \\
         & \entryIIi & \ddots & \vdots     & \vdots   \\
         &           & \ddots & * & * \\
         &           & \cdots & \entryIViv & * \\
         &           & \cdots &            & \entryVv
\end{NiceArray}\right]\\
&=
\begin{bmatrix}
H_1 & * & * & \cdots & * \\
& H_2 & * & \cdots & * \\
& & \ddots & \ddots & \vdots \\
& & & H_i & * \\
& & & & H_{i+1} \\
& & & & - \Sigma_{j=2}^{i+1} S_j|_{i+2\sim\nsharp}^{i+2-j} H_{i+2-j}
\end{bmatrix}.
\end{align*}
By Lemma~\ref{thm: image cap=0}, it follows that
$- \Sigma_{j=2}^{i+1} S_j|_{i+2\sim\nsharp}^{i+2-j} H_{i+2-j} = 0$,
which implies
\begin{equation}\label{eqn: necessity Ti+1~nCAB new}
R_1|_{i+1\sim\nsharp} C \Ao^{i} B
= \begin{bmatrix} H_{i+1} \\ 0|_{i+2\sim\nsharp} \end{bmatrix}.
\end{equation}
Combining \eqref{eqn: necessity Ti+1~nCAB new} with the inductive hypothesis
\eqref{eqn: necessity inductive hypothesis} yields
\begin{equation}\label{eqn: necessity case i+1 TCAB}
R_1|_{j\sim\nsharp} C \Ao^{j-1} B
= \begin{bmatrix} H_j \\ 0|_{j+1\sim\nsharp} \end{bmatrix},
\quad j = 1, 2, \dots, i+1.
\end{equation}
Since $H_j$ for $j = 1, 2, \dots, i$ all possess full row rank,
we deduce that
\begin{equation}\label{eqn: necessary S new=0}
S_j|^{i+2-j}_{i+2\sim\nsharp} = 0, \quad j = 2, 3, \dots, i+1.
\end{equation}
Integrating \eqref{eqn: necessary S new=0} with the inductive hypothesis
\eqref{eqn: necessity inductive hypothesis} results in
\begin{equation}\label{eqn: necessity case i+1 S}
S_j|_{j+k\sim\nsharp}^k = 0, \quad k = 1, 2, \dots, i+2-j, \ j = 2, 3, \dots, i+1.
\end{equation}
By \eqref{eqn: necessary S new=0}, \eqref{eqn: necessity case i+1 S} and
the principle of mathematical induction, the hypothesis
\eqref{eqn: necessity inductive hypothesis} holds for all $i = 1, 2, \dots, \nsharp-1$.
Finally, substituting $i = \nsharp-1$ into the second equation of
\eqref{eqn: necessity inductive hypothesis} gives
\[
S_j|_{j+k\sim\nsharp}^k = 0, \quad
k = 1, 2, \dots, \nsharp-j, \
j = 2, 3, \dots, \nsharp-1,
\]
which is precisely the condition \eqref{eqn: rewrite main result}.
This completes the proof of the necessity of Theorem~\ref{thm: theorem}.

\section{Conclusion and Future Works}

This paper characterizes the set of all possible Luenberger gains $L$ and output transformation matrices $T_y$ that enable an invertible linear system to possess a regular relative degree.
The proposed determination method is straightforward and computationally efficient, as the constraints on the Luenberger gain $L$ are derived by solving a sequence of linear matrix equations.
Furthermore, it is established that the set of all admissible Luenberger gains $L$ constitutes a linear manifold.
In conjunction with \cite{Cao2025rrdref[J]}, these results demonstrate that the row-echelon form is a suitable mathematical tool for investigating the regular relative degree problem in linear time-invariant systems.
Future research will focus on analyzing the degrees of freedom of the solution space associated with the Luenberger gain $L$ and exploring its structural properties.

\bibliography{mybibfile}

@article{Chen21,
	author = {Y. Chen and W. Respondek},
	title = {From {M}orse triangular form of {ODE} control systems
				to feedback canonical form of {DAE} control systems},
	volume = {358},
	year = {2021},
	pages = {8556--8592},
	journal = {Journal of the Franklin Institute},}

@article{Kucera20,
	author = {V. Ku{\v c}era},
	title = {Stability-preserving {M}orse normal form},
	volume = {65},
	number = {12},
	year = {2020},
	pages = {5099--5113},
	journal = {IEEE Transactions on Automatic Control},}

@article{Sepitka25,
	author = {P. {\v S}epitka and R. {\v S}. Hilscher},
	title = {New existence results for conjoined bases of
				singular linear {H}amiltonian systems with given
				{S}turmian properties},
	volume = {707},
	year = {2025},
	pages = {187--224},
	journal = {Linear Algebra and Its Applications},}

@article{Nicolau25,
	author = {F. Nicolau and H. Mounier and {S.-I.} Niculescu},
	title = {Interplay between discretization and controllability
				of linear delay systems: An algebraic viewpoint},
	volume = {713},
	year = {2025},
	pages = {18--73},
	journal = {Linear Algebra and Its Applications},}

@article{Kouri23,
	author = {D. P. Kouri and Z. Hua and M. Udell},
	title = {A greedy {G}alerkin method to efficiently select
				sensors for linear dynamical systems},
	volume = {679},
	year = {2023},
	pages = {275--304},
	journal = {Linear Algebra and Its Applications},}

@article{Guo,
	author = {P. Guo and B. Zhou and Z. Li and L. Zhao},
	title = {On the existence of a regular relative degree of linear time-varying systems by output transformations},
	journal = {Linear Algebra and Its Applications},
	note = {\ (submitted for publication)}
}

@article{Cao2025rrdref[J],
	author = {S. Cao and B. Zhou and G. Duan},
	title = {Echelon form criterion for the existence of a regular relative degree of linear systems},
	journal = {IEEE Control Systems Letters},
	note = {\ (in press)},
	year = {2026},
	doi = {10.1109/LCSYS.2026.3667764}
}

@article{Chen2021,
	author = {Y. Chen and W. Respondek},
	title = {From {M}orse triangular form of {ODE} control systems to feedback canonical form of {DAE} control systems},
	volume = {16},
	year = {2021},
	pages = {8556--8592},
	journal = {Journal of the Franklin Institute},}

@article{Dragan1987,
	author = {V. Dragan and A. Halanay},
	title = {High-gain feedback stabilization of linear systems.},
	volume = {45},
	number = {2},
	year = {1987},
	pages = {549--577},
	journal = {International Journal of Control},}

@article{Fomichev2017DE[J],
	author = {V. V. Fomichev and A. V. Kraev and A. I. Rogovskiy},
	title = {Reduction of systems to a form with relative degree using dynamic output transformation},
	volume = {53},
	number = {5},
	year = {2017},
	pages = {686--700},
	journal = {Differential Equations},
}

@article{Jordan1977,
	author = {D. Jordan and L. Godbout},
	title = {On computation of the canonical pencil of a linear system},
	volume = {22},
	number = {1},
	year = {1977},
	pages = {126--128},
	journal = {IEEE Transactions on Automatic Control},
}

@article{Kucera2020TAC[J],
	author = {V. Ku{\v c}era},
	title = {Stability-preserving {M}orse normal form},
	volume = {65},
	number = {12},
	year = {2020},
	pages = {5099--5113},
	journal = {IEEE Transactions on Automatic Control},}

@conference{Kucera2019ECC[C],
	author = {V. Ku{\v c}era},
	title = {An alternative proof of the {K}ronecker/{M}orse normal form},
	address = {Napoli, Italy},
	year = {2019},
	pages = {3809--3816},
	booktitle = {European Control Conference},}

@conference{Kitapci1984CDC[C],
	author = {A. Kitapci and L. M. Silverman},
	title = {Determination of {M}orses canonical form using the structure algorithm},
	address = {Las Vegas, NV, USA},
	year = {1984},
	pages = {1752--1757},
	booktitle = {IEEE Conference on Decision and Control},}

@article{Morse1973SIAM[J],
	author = {A. S. Morse},
	title = {Structural invariants of linear multivariable systems},
	volume = {11},
	number = {3},
	year = {1973},
	pages = {446--465},
	journal = {SIAM Journal on Control},
}

@article{Mueller2009[J],
	author = {M. Mueller},
	title = {Normal form for linear systems with respect to its vector relative degree},
	volume = {430},
	number = {4},
	year = {2009},
	pages = {1292--1312},
	journal = {Linear Algebra and Its Applications},
}

@article{Patil2021,
	author = {M. Patil and B. Bandyopadhyay and D. Khimani and M. Toe},
	title = {Robust output tracking for the non-minimum phase over-actuated systems},
	volume = {131},
	year = {2021},
	pages = {109726},
	journal = {Automatica},}

@article{Rospondek2017,
	author = {W. Rospondek},
	title = {Right and left invertibility of nonlinear control systems},
	year = {2017},
	pages = {133--176},
	journal = {Nonlinear controllability and optimal
control},
}

@article{Sannuti1987,
	author = {P. Sannuti and A. Saberi},
	title = {Special coordinate basis for multivariable linear systems---finite and infinite zero structure, squaring down
and decoupling},
	volume = {45},
	number = {5},
	year = {1987},
	pages = {1655--1704},
	journal = {International Journal of Control},
}

@article{Suda1994,
	author = {N. Suda},
	title = {An elementary derivation of {K}ronecker canonical form for linear time-invariant systems},
	volume = {27},
	number = {9},
	year = {1994},
	pages = {73-76},
	journal = {IFAC Proceedings Volumes},
}

@article{Sain1969TAC[J],
	author = {Sain, M. and Massey, J.},
	title = {Invertibility of linear time-invariant dynamical systems},
	volume = {14},
	number = {2},
	year = {1969},
	pages = {141--149},
	journal = {IEEE Transactions on Automatic Control},
}

@article{Thorp1973,
	author = {J. S. Thorp},
	title = {The singular pencil of a linear dynamical system},
	volume = {18},
	number = {3},
	year = {1973},
	pages = {577--596},
	journal = {International Journal of Control},}

@article{vanDooren1979LAA[J],
	author = {P. {Van Dooren}},
	title = {The computation of {K}ronecker's canonical form of a singular pencil},
	volume = {27},
	year = {1979},
	pages = {103--140},
	journal = {Linear Algebra and Its Applications},
}

@article{Zhou2023Auto[J],
	author = {B. Zhou},
	title = {On the relative degree and normal forms of linear systems by
	output transformation with applications to tracking},
	volume = {148},
	number = {110800},
	year = {2023},
	journal = {Automatica},
}

@book{Zhoubook,
	author = {B. Zhou},
	title = {Linear Systems Theory},
	note = {(in Chinese)},
	publisher = {Science Press},
	location = {Beijing},
	year = {2024},
}

@article{Zhou2024IJC[J],
	author = {B. Zhou},
	title = {On the invertibility indices and the {M}orse normal form for linear multivariable systems},
	volume = {97},
	number = {6},
	year = {2024},
	journal = {International Journal of Control},
}

\section*{Appendix}

{\renewcommand\proofname{Proof for
  \Cref{thm: invertible and Morse form}}
\begin{proof}
It follows from \eqref{eqn: rank Qi A+LC} and \eqref{eqn: rank Qi TxTuTy} that
the output injection and the algebraic equivalence transformation don't change the rank of the invertibility matrix.
With the fact that a system with a regular relative degree is invertible,
the necessity of this lemma can be easily obtained.

Then we prove the sufficiency.
From Lemma~\ref{thm: Morse normal form}, we know that there exist
$\tilde T_x\in\R^{n\times n}_n$,
$\tilde T_u\in\R^{m\times m}_m$,
$\tilde T_y\in\R^{m\times m}_m$,
$\tilde K\in\R^{m\times n}$ and
$L\in\R^{n\times m}$ such that
$\tilde A=\tilde T_x(A+B\tilde K+LC)\tilde T_x^{-1}$,
$\tilde B=\tilde T_xB\tilde T_u^{-1}$,
$\tilde C=\tilde T_yC\tilde T_x^{-1}$
has the form \eqref{eqn: Morse normal form's ABC}.
Thus, $A+LC=\tilde T_x^{-1}\tilde A\tilde T_x
-\tilde T_x^{-1}\tilde B\tilde T_u\tilde K$,
$B=\tilde T_x^{-1}\tilde B\tilde T_u$,
$C=\tilde T_y^{-1}\tilde C\tilde T_x$.
After the state feedback $\tilde u=\tilde Kx+u$, the system
$(A+LC,B,C)$
becomes $(\tilde T_x^{-1}\tilde A\tilde T_x,
\tilde T_x^{-1}\tilde B\tilde T_u,
\tilde T_y^{-1}\tilde C\tilde T_x)$.
After the algebraic equivalent transformation $\tilde x=\tilde T_x^{-1}x$,
$\tilde u=\tilde T_u^{-1}u$ and $\tilde y=\tilde T_y^{-1}y$,
the system $(\tilde T_x^{-1}\tilde A\tilde T_x,
\tilde T_x^{-1}\tilde B\tilde T_u,
\tilde T_y^{-1}\tilde C\tilde T_x)$ becomes $(\tilde A,\tilde B,\tilde C)$.
These imply that the system $(A+LC,B,C)$ can be converted into the system
$(\tilde A,\tilde B,\tilde C)$ through a state feedback and an algebraic equivalent transformation.
From \eqref{eqn: rank Qi A+BK}, \eqref{eqn: rank Qoc A+BK},
\eqref{eqn: rank Qi TxTuTy} and \eqref{eqn: rank Qoc TxTuTy},
for $i=1,2,\dots,n$, there holds
\begin{equation}\label{eqn: A+LC,BC=tilde ABC}
\begin{gathered}
\Qi^{[i]}(A+LC,B,C)=\Qi^{[i]}(\tilde A,\tilde B,\tilde C),\\
\Qoc^{[i]}(A+LC,B,C)=\Qoc^{[i]}(\tilde A,\tilde B,\tilde C).
\end{gathered}
\end{equation}
It follows from \eqref{eqn: Morse normal form's ABC} that
$\tilde C\tilde A^i\tilde B=C_3A_3^iB_3$ for $i=0,1,\dots,n-1$.
Hence, for $i=1,2,\dots,n$, we have
\begin{equation}\label{eqn: ABC=A_3B_3C_3}
\Qi^{[i]}(\tilde A,\tilde B,\tilde C)=\Qi^{[i]}(A_3,B_3,C_3),\
\Qoc^{[i]}(\tilde A,\tilde B,\tilde C)=\Qoc^{[i]}(A_3,B_3,C_3).
\end{equation}
One can easily verify that the system $(A_3,B_3,C_3)$ satisfies
\eqref{eqn: Zhou2023(8)(9)}.
Then, it follows from \eqref{eqn: A+LC,BC=tilde ABC} and \eqref{eqn: ABC=A_3B_3C_3} that
the system $(A+LC,B,C)$ satisfies \eqref{eqn: Zhou2023(8)(9)}.
From Lemma~\ref{thm: Zhou2023(8)(9)},
there exists a $T_y\in\R^{m\times m}_m$ such that the system
$(A+LC,B,T_yC)$ has a regular relative degree.
This finishes the proof.
\end{proof}}

{\renewcommand\proofname{Proof for \Cref{thm: ei=d'i}}
\begin{proof}
Proof for Item~1:
For $i=1,2,\dots,n-1$, on one hand, \eqref{eqn:rrefDniswhat} can be rewritten as
\[
\rref(D^{[n]})
=\left[\begin{array}{ccc|ccc}
E_1&\cdots&*&*&\cdots&*\\
&\ddots&\vdots&\vdots&\ddots&\vdots\\
&&E_i&*&\cdots&*\\ \hline
&&&E_{i+1}&\cdots&*\\
&&&&\ddots&\vdots\\
&&&&&E_{n}\\
\multicolumn{6}{c}{0}
\end{array}\right].
\]
On the other hand, the reduced row-echelon of $D^{[n]}$ can be written as
\[
\rref(D^{[n]})
=\rref\begin{bmatrix}&D^{[n-i]}\\ D^{[i]}&*\end{bmatrix}
=\rref\begin{bmatrix}D^{[i]}&*\\ &D^{[n-i]}\end{bmatrix}.
\]
With the property of the reduced row-echelon form,
we know that \eqref{eqn: Dirref=E1~i} holds.

Proof for Item~2:
From the \Cref{thm: Drref=E1~i} of \Cref{thm: ei=d'i}, we know that
$d_i=\Sigma_{j=1}^ie_j$ for $i=1,2,\dots,n$.
Thus, $e_i=d_i'$.

Proof for Item~3:
For $i=2,\dots,n$, suppose that $P_i\in\R^{mi\times mi}_{mi}$
is the matrix transforming $D^{[i]}$
to its reduced row-echelon form,
that is, $P_iD^{[i]}=\rref(D^{[i]})$.
Note that
$
D^{[i]}=\begin{bmatrix}&D^{[i-1]}\\ D_1&*\end{bmatrix}
$.
Thus
\begin{align}\label{eqn: E1<E2<...<En}
\begin{bmatrix}&I_m\\ P_{i-1}&\end{bmatrix}D^{[i]}
\notag&=\begin{bmatrix}&I_m\\ P_{i-1}&\end{bmatrix}
\begin{bmatrix}&D^{[i-1]}\\ D_1&*\end{bmatrix}
=\begin{bmatrix}D_1&*\\ &P_{i-1}D^{[i-1]}\end{bmatrix}\\
&=\begin{bmatrix}D_1&*\\ &\rref(D^{[i-1]})\end{bmatrix}
=\begin{bmatrix}
D_1&*&\cdots&*\\
&E_1&\cdots&*\\
&&\ddots&\vdots\\
&&&E_{i-1}\\
\multicolumn{4}{c}{0}
\end{bmatrix}.
\end{align}
By combining \eqref{eqn: Dirref=E1~i} with \eqref{eqn: E1<E2<...<En},
we have $\imr{E_{i-1}}\subseteq\imr{E_i}$.

Proof for Item~4:
It follows from the \Cref{item:E1<En} and $e_i=\im(E_i^\T),i=1,2,\dots,n$ that
$e_1\le e_2\le\cdots\le e_n$.
With the Item~2, we know that this item holds.
This finishes the proof.
\end{proof}}

{\renewcommand\proofname{Proof for \Cref{thm: QD=F}}
\begin{proof}
Proof for Item 1:
For $j=1,2,\dots,i$, since $\imr{F_j}\subseteq\imr{E_j}$,
there exists coefficient matrices
$K_j\in\R^{f_j\times d'_j}$ such that $F_j=K_jE_j$.
Denote $K=K_1\oplus K_2\oplus\cdots\oplus K_i$.
With the \Cref{thm: Drref=E1~i} of \Cref{thm: ei=d'i}, there
exists a $Q_{\rref}\in\R^{d_i\times mi}$ such that
\[
Q_{\rref}D^{[i]}
=\rref(D^{[i]})
=\begin{bmatrix}
E_1&*&\cdots&*\\
&E_2&\cdots&*\\
&&\ddots&\vdots\\
&&&E_i
\end{bmatrix}.
\]
Therefore,
\[
KQ_{\rref}D^{[i]}=
\begin{bmatrix}
K_1E_1&*&\cdots&*\\
&K_2E_2&\cdots&*\\
&&\ddots&\vdots\\
&&&K_iE_i
\end{bmatrix}
=\begin{bmatrix}
F_1&*&\cdots&*\\
&F_2&\cdots&*\\
&&\ddots&\vdots\\
&&&F_i
\end{bmatrix}.
\]
We can just choose $Q=KQ_{\rref}$.

Proof for Item 2:
We will use the following simple result of the linear algebra theory:
\begin{equation*}
\begin{minipage}{\linewidth}
Suppose that $\mathcal V_1,\mathcal V_2$ are two subspaces of $\R^m$
satisfying $\mathcal V_1\subseteq\mathcal V_2$.
For two matrices $X_1\in\R^{m\times\dim(\mathcal V_1)}_{\dim(\mathcal V_1)}$,
$X_2\in\R^{m\times[\dim(\mathcal V_2)-\dim(\mathcal V_1)]}_{\dim(\mathcal V_2)-\dim(\mathcal V_1)}$,
if $\im(X_1)\subseteq\mathcal V_1$ and $\im(X_2)\backslash\{0\}\subseteq\mathcal V_2\backslash\mathcal V_1$,
then $\im\begin{bmatrix}X_1&X_2\end{bmatrix}=\mathcal V_2$.
\end{minipage}
\end{equation*}
With the above fact,
using the induction method,
there holds the following fact:
\begin{equation*}
\begin{minipage}{\linewidth}
For $i\ge2$,
suppose that $\mathcal V_0=\{0\}$ and $\mathcal V_1,\mathcal V_2,\dots,\mathcal V_i$ are a sequence of subspaces of $\R^m$
satisfying $\mathcal V_1\subseteq\mathcal V_2\subseteq\cdots\subseteq\mathcal V_i$.
For a sequence of matrices $X_j\in
\R^{m\times\nabla\!\dim(\mathcal V_j)},j=1,2,\dots,i$,
if $\im(X_j)\backslash\{0\}\subseteq\mathcal V_j\backslash\mathcal V_{j-1}$
for $j=1,2,\dots,i$,
then $\im\begin{bmatrix}X_1&X_2&\cdots&X_i\end{bmatrix}=\mathcal V_i$.
\end{minipage}
\end{equation*}
With the above fact,
one can easily prove this item.

Proof for Item 3:
The proof is divided into three steps.

\emph{Step 1:}
To prove that $  Q_1$ is of full row rank,
we just need to prove that, for a
coefficient row vector $k\in\R^{1\times  f}$,
if $k  Q_1=0$, then $k=0$.
Suppose that $Q=\begin{bmatrix}\bar Q_1&Q_1\end{bmatrix}$, $k  Q_1=0$ and
let $k=\begin{bmatrix}k_1&k_2&\cdots&k_i\end{bmatrix}$
where $k_j\in\R^{1\times  f_j}$ for $j=1,2,\dots,i$.
On one hand, with $k  Q_1=0$, we get
\begin{equation}\label{eqn: (kQ)D=?}
\begin{aligned}
kQD^{[i]}
&=k\begin{bmatrix}\bar Q_1&Q_1\end{bmatrix}D^{[i]}
=\begin{bmatrix}k\bar Q_1&kQ_1\end{bmatrix}D^{[i]}\\
&=\begin{bmatrix}k\bar Q_1&0\end{bmatrix}
\begin{bmatrix}&D^{[n-1]}\\ D_1&*\end{bmatrix}
=\begin{bmatrix}0&k\bar Q_1  D^{[n-1]}\end{bmatrix}.
\end{aligned}
\end{equation}
On the other hand, with the \Cref{thm:existQ}, we have
\begin{equation}\label{eqn: k(QD)=?}
k(QD^{[i]})=
\begin{bmatrix}k_1&*\end{bmatrix}
\left[\begin{array}{c|ccc}
F_1&*&\cdots&*\\
\hline
&F_2&\cdots&*\\
&&\ddots&\vdots\\
&&&F_i
\end{array}\right]
=\begin{bmatrix}k_1F_1&*\end{bmatrix}.
\end{equation}
By comparing \eqref{eqn: (kQ)D=?} and \eqref{eqn: k(QD)=?},
we know that $k_1F_1=0$.

\emph{Step 2:}
If $i=1$, then this Item is proven.
In the sequel, we consider the cases $2\le i\le n$.
By noting that $F_1$ is of full row rank, we have $k_1=0$.
Since the \Cref{thm: Drref=E1~i} of \Cref{thm: ei=d'i} reveals
the reduced row-echelon form of $D^{[i-1]}$,
there exists a nonsingular matrix $Q_{\rref}^{[i-1]}\in\R^{m(i-1)\times m(i-1)}_{m(i-1)}$ such that
\begin{equation}\label{eqn: Qtilde D[i-1]=rref D[i-1]}
Q_{\rref}^{[i-1]}D^{[i-1]}
=\rref(D^{[i-1]})
=\begin{bmatrix}
E_1&\cdots&*\\
&\ddots&\vdots\\
&&E_{i-1}\\
&0&
\end{bmatrix}.
\end{equation}
Define $l\in\R^{1\times d_{i-1}}$ as $k\bar Q_1(Q_{\rref}^{[i-1]})^{-1}=
\begin{bmatrix}l&*\end{bmatrix}$.
Then from \eqref{eqn: Qtilde D[i-1]=rref D[i-1]}, we can obtain
\begin{align*}
k\bar Q_1D^{[i-1]}
&=[k\bar Q_1(Q_{\rref}^{[i-1]})^{-1}](Q_{\rref}^{[i-1]}D^{[i-1]})
\\
&
=\begin{bmatrix}l&*\end{bmatrix}
\begin{bNiceMatrix}[margin=3.5pt]
E_1 & \cdots & * \\
    & \ddots & \vdots \\
    &        & E_{i-1} \\
\Hline
    & 0      &
\end{bNiceMatrix}
=l\begin{bmatrix}
E_1&\cdots&*\\
&\ddots&\vdots\\
&&E_{i-1}
\end{bmatrix}.
\end{align*}
Let $l=\begin{bmatrix}l_1&l_2\cdots&l_{i-1}\end{bmatrix}$
where $l_j\in\R^{1\times d'_j}$ for $j=1,2,\dots,i-1$.
Then we have
\begin{equation}\label{eqn: kQbarD[n-1]=lE1~i-1}
k\bar Q_1D^{[n-1]}
=\begin{bmatrix}l_1&\cdots&l_{i-1}\end{bmatrix}
\begin{bmatrix}
E_1&\cdots&*\\
&\ddots&\vdots\\
&&E_{i-1}
\end{bmatrix}.
\end{equation}
By recalling \eqref{eqn: k(QD)=?}, with $k_1=0$, we have
\begin{equation}\label{eqn: kQD=[0 kF2~i]}
\begin{aligned}
k(QD^{[i]})
&=\left[\begin{array}{c|ccc}0&k_2&\cdots&k_i\end{array}\right]
\left[\begin{array}{c|ccc}
F_1&*&\cdots&*\\
\hline
&F_2&\cdots&*\\
&&\ddots&\vdots\\
&&&F_i
\end{array}\right]\\
&=\begin{bmatrix}0&
\begin{bmatrix}
k_2&\cdots&k_i
\end{bmatrix}
\begin{bmatrix}
F_2&\cdots&*\\
&\ddots&\vdots\\
&&F_i
\end{bmatrix}
\end{bmatrix}.
\end{aligned}
\end{equation}
It follows from \eqref{eqn: (kQ)D=?}, \eqref{eqn: kQbarD[n-1]=lE1~i-1}
and \eqref{eqn: kQD=[0 kF2~i]} that
\begin{equation}\label{eqn: kF2~i=lE1~i-1}
\begin{bmatrix}k_2&\cdots&k_i\end{bmatrix}
\begin{bmatrix}
F_2&\cdots&*\\
&\ddots&\vdots\\
&&F_i
\end{bmatrix}
=\begin{bmatrix}l_1&\cdots&l_{i-1}\end{bmatrix}
\begin{bmatrix}
E_1&\cdots&*\\
&\ddots&\vdots\\
&&E_{i-1}
\end{bmatrix}.
\end{equation}
By rewriting \eqref{eqn: kF2~i=lE1~i-1} as
$
\begin{bmatrix}k_2&*\end{bmatrix}
\begin{bmatrix}F_2&*\\ &*\end{bmatrix}
=\begin{bmatrix}l_1&*\end{bmatrix}
\begin{bmatrix}E_1&*\\ &*\end{bmatrix}
$,
there holds $\begin{bmatrix}k_2F_2&*\end{bmatrix}
=\begin{bmatrix}l_1E_1&*\end{bmatrix}$.
Thus, $k_2F_2=l_1E_1$.
Since $\imr{F_2}\cap\imr{E_1}=\{0\}$,
there holds $k_2F_2=l_1E_1=0$.
Since $F_2$ and $E_1$ are both of full row rank,
it follows that $k_2=0$ and $l_1=0$.
This finishes the proof for case $i=2$.

\emph{Step 3:}
For cases $i\ge3$,
we make the inductive hypothesis that, for $j=2,3,\dots,i-1$,
there holds
\begin{equation}\label{eqn: inductive hypothesis k2~j=0 l1~j-1=0}
k_2=0,\ k_3=0,\ \cdots,\ k_j=0
\text{ and }l_1=0,\ l_2=0,\ \cdots,\ l_{j-1}=0.
\end{equation}
Then from \eqref{eqn: kF2~i=lE1~i-1}, we can get
\[
\begin{bmatrix}k_{j+1}&\cdots&k_i\end{bmatrix}
\begin{bmatrix}
F_{j+1}&\cdots&*\\
&\ddots&\vdots\\
&&F_i
\end{bmatrix}
=\begin{bmatrix}l_j&\cdots&l_{i-1}\end{bmatrix}
\begin{bmatrix}
E_j&\cdots&*\\
&\ddots&\vdots\\
&&E_{i-1}
\end{bmatrix}.
\]
The above equation can be rewritten as
$
\begin{bmatrix}k_{j+1}&*\end{bmatrix}
\begin{bmatrix}F_{j+1}&*\\ &*\end{bmatrix}
=\begin{bmatrix}l_j&*\end{bmatrix}
\begin{bmatrix}E_j&*\\ &*\end{bmatrix}
$.
Thus $\begin{bmatrix}k_{j+1}F_{j+1}&*\end{bmatrix}
=\begin{bmatrix}l_jE_j&*\end{bmatrix}$ and then we have
$k_{j+1}F_{j+1}=l_jE_j$.
Since $\imr{F_{j+1}}\cap\imr{E_j}=\{0\}$,
there holds $k_{j+1}F_{j+1}=l_jE_j=0$.
Since $F_{j+1}$ and $E_j$ are both of full row rank,
it follows that $k_{j+1}=0$ and $l_j=0$.
Together with the inductive hypothesis
\eqref{eqn: inductive hypothesis k2~j=0 l1~j-1=0}, there hold
\[
k_2=0,\ k_3=0,\ \dots,\ k_{j+1}=0
\text{ and }l_1=0,\ l_2=0,\ \dots,\ l_j=0.
\]
By using the induction method,
it follows that $k_2=0$, $k_3=0$, $\dots$, $k_i=0$.
Therefore, we can conclude that $k=0$.
This finishes the proof.
\end{proof}}

{\renewcommand\proofname{Proof for Remark~\ref{rmk: n* and nsharp}}
\begin{proof}
Proof for Item~1:
It follows from \eqref{eqn: Zhou2023(8)(9)} that
$\rank(\Qi^{[i]})=\Sigma_{j=1}^i\rank(\Qoc^{[j]})$ for
$i=1,2,\dots,n$.
Then, $\faii^{[i]}=\rank(\Qoc^{[i]})$ for $i=1,2,\dots,n$.
With the definitions of $\nsharp$ and $\nocstar$, it follows that $\nsharp=\nocstar$.

Proof for Item~2:
It follows from \eqref{eqn: rank Qi A+LC} that
$\nsharp(A+LC,B,C)=\nsharp(A,B,C)$.
Since the system $(A+LC,B,C)$ satisfies \eqref{eqn: Zhou2023(8)(9)},
then from the Item 1,
we have $\nsharp(A+LC,B,C)=\nocstar(A+LC,B,C)$.
Thus $\nocstar(A+LC,B,C)=\nsharp(A,B,C)$.

Proof for Item~3:
To prove $\nsharp\ge\nocstar$, we just need to prove the fact:
\[
\begin{minipage}{\linewidth}
For $i=1,2,\dots,n$, if
$\faii^{[i]}=m$,
then $\rank(\Qoc^{[i]})=m$.
\end{minipage}
\]
Due to the definition of $\Qi^{[i]}$,
we get
$\rank(\Qi^{[i]})\le\rank(\Qi^{[i-1]})+\rank(\Qoc^{[i]})$.
Thus, $\rank(\Qoc^{[i]})\ge m$.
Since the row number of $\Qoc^{[i]}$ is $m$,
we know $\rank(\Qoc^{[i]})=m$.
This finishes the proof of the above fact.

Proof for Item~4 (\Cref{thm: if nsharp=2 then n*=2}):
From Item~3, if $\nsharp=2$, then necessarily $\nocstar\in\{1,2\}$.
Moreover, since $\Qoc^{[1]}=\Qi^{[1]}$, the case $\nocstar=1$ implies $\nsharp=1$.
Hence, $\nocstar=2$.
This finishes the proof.
\end{proof}}

\end{document}